\documentclass[11pt]{amsart}
\usepackage{a4wide,amssymb}
\usepackage[all]{xy}

\newcommand{\IF}{\mathbb F}
\newcommand{\IQ}{\mathbb Q}
\newcommand{\IN}{\mathbb N}
\newcommand{\IR}{\mathbb R}
\newcommand{\IC}{\mathbb C}

\newcommand{\w}{\omega}

\newcommand{\IZ}{\mathbb Z}
\newcommand{\pr}{\mathrm{pr}}
\newcommand{\Fix}{\mathrm{Fix}}

\newtheorem{theorem}{Theorem}
\newtheorem{corollary}{Corollary}
\newtheorem{lemma}{Lemma}
\newtheorem{claim}{Claim}
\newtheorem{proposition}{Proposition}
\newtheorem{problem}{Problem}
\newtheorem{conjecture}{Conjecture}

\theoremstyle{definition}
\newtheorem{definition}{Definition}
\newtheorem{example}{Example}
\newtheorem{remark}{Remark}

\title{A universal coregular countable second-countable space}
\author{Taras Banakh and Yaryna Stelmakh}
\address{T.Banakh: Ivan Franko National University of Lviv (Ukraine), and Jan Kochanowski University in Kielce (Poland)}
\email{t.o.banakh@gmail.com}
\address{Y. Stelmakh: Ivan Franko National University of Lviv (Ukraine)}
\email{yarynziya@ukr.net}

\begin{document}
\begin{abstract} A Hausdorff topological space $X$ is called {\em superconnected} (resp. {\em coregular}) if for any nonempty open sets $U_1,\dots U_n\subseteq X$, the intersection of their closures $\overline U_1\cap\dots\cap\overline U_n$ is not empty (resp. the complement $X\setminus (\overline U_1\cap\dots\cap\overline U_n)$ is a regular topological space). A canonical example of a coregular superconnected space is the projective space $\mathbb Q\mathsf P^\infty$ of the topological vector space $\mathbb Q^{<\omega}=\{(x_n)_{n\in\omega}\in \mathbb Q^{\omega}:|\{n\in\omega:x_n\ne 0\}|<\omega\}$ over the field of rationals $\mathbb Q$. 
The space $\mathbb Q\mathsf P^\infty$ is the quotient space of $\mathbb Q^{<\omega}\setminus\{0\}^\omega$ by the equivalence relation $x\sim y$ iff $\mathbb Q{\cdot}x=\mathbb Q{\cdot}y$.

We prove that every countable second-countable coregular space is homeomorphic to a subspace of $\mathbb Q\mathsf P^\infty$, and a topological space $X$ is homeomorphic to 
$\mathbb Q\mathsf P^\infty$ if and only if $X$ is countable, second-countable, and admits a decreasing sequence of closed sets $(X_n)_{n\in\omega}$ such that (i) $X_0=X$, $\bigcap_{n\in\omega}X_n=\emptyset$, (ii) for every $n\in\omega$ and a nonempty open set $U\subseteq X_n$ the closure $\overline U$ contains some set $X_m$, and (iii) for every $n\in\omega$ the complement $X\setminus X_n$ is a regular topological space. Using this topological characterization of $\mathbb Q\mathsf P^\infty$ we find topological copies of the space $\mathbb Q\mathsf P^\infty$ among quotient spaces, orbit spaces of group actions, and projective spaces of topological vector spaces over countable topological fields.
\end{abstract}
\maketitle

\section{Introduction}

A topological space $X$ is called {\em functionally Hausdorff\/} if for any distinct points $x,y\in X$ there exists a continuous function $f:X\to\IR$ such that $f(x)\ne f(y)$. 
It is easy to see that each countable functionally Hausdorff space is (totally) disconnected. On the other hand, there are many examples of countable Hausdorff spaces which are connected and even superconnected. A topological space $X$ is defined  to be {\em superconnected} if for any nonempty open sets $U_1,\dots,U_n$ in $X$ the intersection of their closures $\overline U_{\!1}\cap\dots\cap\overline U_{\!n}$ is not empty. It is easy to see that continuous images of superconnected spaces remain superconnected.

 One of standard examples of a superconnected Hausdorff space is the famous   Golomb space, introduced by Brown \cite{Brown} and popularized by Golomb \cite{Golomb59}, \cite{Golomb61}. The {\em Golomb space} is the space $\IN$ of positive integer numbers, endowed with the topology generated by the base consisting of the arithmetic progressions $\IN\cap (a+b\IZ)$ with relatively prime numbers $a,b\in\IN$. Using the Chinese remainder theorem, it can be shown (see \cite{BMT}) that the closure of the arithmetic progression $\IN\cap(a+b\IZ)$ in the Golomb space contains the arithmetic progression $p_1\cdots p_n\IN$, where $p_1,\dots,p_n$ are prime divisors of $b$. This implies that the Golomb space is superconnected. Then any continuous image of the Golomb space is superconnected as well. One of such images is the {\em Kirch space} \cite{Kirch}, which is the space $\IN$ endowed with the topology generated by the base consisting of the arithmetic progressions $\IN\cap (a+b\IZ)$ where the numbers $a,b$ are relatively prime and $b$ is not divided by a square of a prime number. The Kirch space is known to be superconnected and locally connected. In \cite{BST} and \cite{BYT} it was shown that the Golomb and Kirch spaces are topologically rigid, i.e, have trivial homeomorphism group.
 
 Another natural example of a superconnected countable Hausdorff space is the {\em infinite rational projective space} $\IQ \mathsf P^\infty$, which is the projective space of the topological vector space $$\IQ^{<\w}=\{(x_n)_{n\in\w}\in\IQ^\w:|\{n\in\w:x_n\ne 0\}|<\w\}$$over the field $\IQ$ of rational numbers. Here the countable power $\IQ^\w$ is endowed with the Tychonoff product topology. The space $\IQ \mathsf P^\infty$ is defined as the quotient space of the space $\IQ^{<\w}_\circ=\IQ^{<\w}\setminus\{0\}^\w$ by the equivalence relation $x\sim y$ iff $x=\lambda y$ for some nonzero rational number $\lambda$. The superconnectedness of the rational projective space $\IQ\mathsf P^\infty$ was first noticed by Gelfand and Fuks in their paper \cite{GF}. It is easy to show (see also Theorem~\ref{t:hom}) that the infinite rational projective space $\IQ\mathsf P^\infty$ is topologically homogeneous (i.e., for any points $x,y\in \IQ \mathsf P^\infty$ there exists a homeomorphism $h$ of $\IQ \mathsf P^\infty$ such that $h(x)=y$). So, $\IQ \mathsf P^\infty$ is not homeomorphic to the Golomb or Kirch space. Another important property that distinghished the space $\IQ\mathsf P^\infty$ from the Golomb and Kirch spaces is the coregularity of $\IQ\mathsf P^\infty$.
 
A topological space $X$ is called {\em coregular} if $X$ is Hausdorff and for any nonempty open sets $U_1,\dots,U_n\subset X$ the complement $X\setminus (\overline U_{\!1}\cap\dots\cap \overline U_{\!n})$ is a regular topological space.  We recall that a topological space $X$ is {\em regular} if it is Hausdorff and for every closed set $F\subset X$ and point $x\in X\setminus F$ there are disjoint open sets $U,V$ in $X$ such that $x\in U$ and $F\subseteq V$. It is easy to see that a topological space containing more than one point  is regular if and only if it is coregular and not superconnected. In Theorems~\ref{p3} and \ref{t:univ} we shall prove that the space $\IQ\mathsf P^\infty$ is coregular and moreover, it is contains a topological copy of each coregular countable second-countable space. Let us recall that a topological space $X$ is {\em second-countable} if it has a countable base of the topology. 

Observe that for a coregular topological space $X$ with a countable base of the topology $\{U_n\}_{n\in\w}$, the sequence $(X_n)_{n\in\w}$ of the sets $X_n=\overline U_{\!0}\cap\dots\cap \overline U_{\!n}$ has the following two properties: (i) for any nonempty open set $U\subseteq X$ the closure $\overline{U}$ contains some set $X_n$, and (ii) for every $n\in\w$ the complement $X\setminus X_n$ is a regular topological space. This property of the sequence $(X_n)_{n\in\w}$ motivates the following definition.

\begin{definition}\label{d:skeleton} A sequence $(X_n)_{n\in\w}$ of {\em closed} subsets of a topological space $X$ is called
\begin{itemize}
\item {\em vanishing} if $X_0=X$, $\bigcap_{n\in\w}X_n=\emptyset$ and $X_{n+1}\subseteq X_n$ for every $n\in\w$;
\item {\em a coregular skeleton} for $X$ if $(X_n)_{n\in\w}$ is vanishing and has two properties: (i) for every nonempty open set $U\subseteq X$ the closure $\overline U$ contains some set $X_n$ and (ii) for every $n\in\w$ the complement $X\setminus X_n$ is a regular topological space;
\item {\em a superconnecting skeleton} for $X$ if $(X_n)_{n\in\w}$ is vanishing and for every nonempty open set $U\subseteq X$ there exists $n\in\w$ such that $\emptyset\ne  X_n\subseteq\overline U$;
\item {\em an inductively superconnecting skeleton} for $X$ if $(X_n)_{n\in\w}$ is vanishing and for every $n\in\w$ and nonempty open set $U\subseteq X_n$ there exists $m\in\w$ such that $\emptyset\ne  X_m\subseteq\overline U$;
\item {\em superskeleton} if it is both a coregular skeleton and an inductively superconnecting skeleton;
\item {\em canonical superskeleton} if  $(X_n)_{n\in\w}$ is a superskeleton and for every $n\in\w$ the set $X_{n+1}$ is nowhere dense in $X_n$.
\end{itemize}
\end{definition}
It is clear that a topological space $X$ is coregular  if it has a coregular skeleton, and $X$ is superconnected of it has a superconnecting skeleton.

The principal result of this paper is the following topological characterization of the infinite rational projective space $\IQ\mathsf P^\infty$.

\begin{theorem}\label{t:main} A topological space  is homeomorphic to the space $\IQ\mathsf P^\infty$ if and only if it is countable, second-countable and possesses a superskeleton.
\end{theorem}

The proof of Theorem~\ref{t:main} will be presented in Section~\ref{s:main}. Since the proof is long and technical, we postpone it till the end of the paper, and first we apply Theorem~\ref{t:main} to finding topological copies of the space $\IQ\mathsf P^\infty$ among  quotient spaces of topological spaces by equivalence relations and orbit spaces of group actions.

\section{Topological copies of the space $\IQ\mathsf P^\infty$ in ``nature''}

Let $E$ be an equivalence relation on a topological space $X$. A subset $A\subseteq X$ is called  {\em $E$-saturated} if $A$ coincides with its {\em $E$-saturation} $EA=\bigcup_{x\in A}Ex$ where $Ex=\{y\in X:(x,y)\in E\}$ is the {\em equivalence class} of a point $x\in X$. Let $X/E$ be the space of $E$-equivalence classes $\{Ex:x\in X\}$ and $q:X\to X/E$ be the map assigning to each point $x\in X$ its equivalence class $Ex\in X/E$. The space $X/E$ carries the quotient topology, consisting of all subsets $U\subseteq X/E$ whose preimage $q^{-1}(U)$ is open in $X$. 

\begin{proposition}\label{p1} Let $E\subseteq X\times X$ and equivalence relation on a topological space $X$ satisfying the following conditions:
\begin{enumerate}
\item the set $E$ is closed in $X\times X$;
\item for any open set $U\subseteq X$ its $E$-saturation $EU$ is open in $X$;
\item $X$ admits a vanishing sequence $(X_n)_{n\in\w}$ of non-empty $E$-saturated closed sets  such that 
\begin{enumerate}
\item for any $n\in\w$ and  nonempty $E$-saturated open set $U\subseteq X_n$, the closure $\overline{U}$ contains some set $X_m$;
\item for any $n\in\w$, $E$-saturated open set $U\subseteq X$ and $x\in U$, there exists an open $E$-saturated neighborhood $V\subseteq X$ of $x$ such that $\overline{V}\subseteq U\cup X_n$.
\end{enumerate}
\end{enumerate} 
Then the quotient space $X/E$ possesses a superskeleton. If the space $X$ is first-countable and $X/E$ is countable, then $X/E$ is homeomorphic to $\IQ\mathsf P^\infty$.
\end{proposition}

\begin{proof} Let $Y$ be the quotient space $X/E$. The condition (2) implies that the quotient map $q:X\to Y$ is open. Then for any $E$-saturated closed set $A\subseteq X$ its image $q(A)=Y\setminus q(X\setminus A)$ is closed in $Y$. In particular, for every $n\in\w$ the image $Y_n=q(X_n)$ is a closed subset of $Y$ and hence $(Y_n)_{n\in\w}$ is a vanishing sequence of  nonempty closed sets in $Y$. 

For any $n\in\w$ and nonempty open set $U\subseteq Y_n$ the preimage $q^{-1}(U)$ is an open $E$-saturated set in $X_n$ and by condition (3a), the closure $\overline{q^{-1}(U)}$ contains some set $X_m$. Then $Y_m=q(X_m)\subseteq q(\overline{q^{-1}(U)})\subseteq \overline{q(q^{-1}(U))}=\overline{U}$. This means that  $(Y_n)_{n\in\w}$ is an inductively superconnecting skeleton for the space $Y$.  By \cite[2.4.C]{Eng}, the condition (1,2)  imply that the  the quotient space $Y=X/E$ is Hausdorff. Now the condition (3a),(3b) ensure that the skeleton $(Y_n)_{n\in\w}$ is coregular and hence $(Y_n)_{n\in\w}$ is a superskeleton for $X/E$.
 
If the space $X$ is first-countable and $X/E$ is countable, then the openness of the quotient map $q:X\to X/E$ implies the first-countability of the the quotient space $X/E$. Being countable and first-countable, the space $X/E$ is second-countable. By Theorem~\ref{t:main}, the space $X/E$ is homeomorphic to $\IQ\mathsf P^\infty$. 
\end{proof}

Next, we consider orbit spaces of group actions, which are special examples of quotient spaces. By a {\em group act} we understand a topological space $X$ endowed with an action $\alpha:G\times X\to X$ a group $G$. The action $\alpha$ satisfies the following axioms:
\begin{itemize}
\item for every $g\in G$ the map $\alpha(g,\cdot):X\to X$, $\alpha(g,\cdot):x\mapsto gx:=\alpha(g,x)$, is a homeomorphism of $X$;
\item for the identity $1_G$ of the group $G$ and every $x\in X$ we have $1_Gx=x$;
\item $(gh)x=g(hx)$ for all $g,h\in G$ and $x\in X$.
\end{itemize}
In this case we also say that $X$ is a {\em $G$-space}. We say that a $G$-space $X$  has {\em closed orbits} if for any point $x\in X$ its {\em orbit} $Gx=\{gx:g\in G\}$ is a closed subset of $X$. A subset $A\subseteq X$ is called {\em $G$-invariant} if it coincides with its $G$-saturation $GA=\bigcup_{x\in A}Gx$.
The action of $G$ on $X$ induces the equivalence relation $E=\{(x,gx):x\in X,\;g\in G\}$. The quotient space $X/E$ by this equivalence relation is called the {\em orbit space} of the $G$-space and is denoted by $X/G$. 

\begin{proposition}\label{p2} Let $X$ be a $G$-space with closed $G$-orbits, possessing a  vanishing sequence $(X_n)_{n\in\w}$ of nonempty $G$-invariant closed subsets  such that 
\begin{enumerate}
\item for any $n\in\w$ and nonempty open $G$-invariant set $U\subseteq X_n$, the closure $\overline{U}$ contains some set $X_m$;
\item for any $n\in\w$, point $x\in X\setminus X_n$, and open $G$-invariant neighborhood $U\subseteq X$ of $x\in U$,  there exists an open $G$-invariant neighborhood $V\subseteq X$ of $x$ such that $\overline{V}\subseteq U\cup X_n$.
\end{enumerate} 
Then the orbit space $X/G$ has a superskeleton. If $X$ is first-countable and $X/G$ is countable, then the space $X/G$ is homeomorphic to $\IQ\mathsf P^\infty$.
\end{proposition}

\begin{proof} The closedness of orbits implies that the orbit space $Y=X/G$ is a $T_1$-space. Since each open set $U\subseteq X$ has open $G$-saturation $GU=\bigcup_{g\in G}gU$, the quotient map $q:X\to X/G$ is open. Then for every $n\in\w$,  the image $Y_n=q(X_n)$ of the closed $G$-invariant set $X_n$ is a closed subset of  $Y$ and hence $(Y_n)_{n\in\w}$ is a vanishing sequence of nonempty closed sets in $Y$. The condition (2) implies that for every $n\in\w$, the open subspace $Y\setminus Y_n$ of the $T_1$-space $Y$ is regular and hence Hausdorff. Since $\bigcap_{n\in\w}Y_n=\emptyset$, the space $Y$ is Hausdorff and the equivalence relation $$E=\{(x,gx):x\in X,\;g\in G\}=\{(x,y)\in X\times X:q(x)=q(y)\}$$ is closed in $X\times X$. Now we can apply Proposition~\ref{p1} and conclude that $(Y_n)_{n\in\w}$ is a superskeleton in the space $X/G$ (and the space $X/G$ is homeomorphic to $\IQ\mathsf P^\infty$ if  $X$ is first-countable and $X/G$ is countable).
\end{proof}

Now we find topological copies of the space $\IQ\mathsf P^\infty$ among infinite projective spaces of singular $G$-spaces.

\begin{definition}\label{d:singular} A  topological space $X$ endowed with a continuous action $\alpha:G\times X\to X$ of a Hausdorff topological group $G$ is called {\em singular} if it has the following properties:
\begin{itemize}
\item[(i)] the topological space $X$ is regular and infinite;
\item[(ii)] the set $\Fix_G(X)=\{x\in X:Gx=\{x\}\}$ is a singleton;
\item[(iii)] for every $x\in X\setminus\Fix_G(X)$ the map $\alpha_x:G\to X$, $\alpha_x:g\mapsto gx=\alpha(g,x)$, is injective and open;
\item[(iv)] the orbit $Gx$ of every point $x\in X\setminus\Fix_G(X)$ contains the singleton $\Fix_G(X)$ in its closure $\overline{Gx}$;
\item[(v)] for any points $x\in X\setminus \Fix_G(X)$ and $y\in X$, there exists a neighborhood $U\subseteq X$ of $y$ such that for any neighborhood $W\subseteq X$ of the singleton $\Fix_G(X)$, there exists a neighborhood $V\subseteq X$ of $\Fix_G(X)$ such that $\alpha_u(\alpha_x^{-1}(V))\subseteq W$ for every $u\in U$.
\end{itemize}
\end{definition}


\begin{example}\label{ex1} There are many natural examples of singular $G$-spaces: 
\begin{enumerate}
\item The complex plane $\IC$ endowed with the action of the multiplicative group $\IC^*$ of non-zero complex numbers.
\item Any subfield $\IF\subseteq\mathbb C$ endowed with the action of the multiplicative group $\IF^*=\IF\setminus\{0\}$.
\item Ahe real line $\IR$ endowed with the action of the multiplicative group $\IR^*$ of non-zero real numbers.
\item The real line $\IR$ endowed with the action of the multiplicative group $\IR_+$ of positive real numbers.
\item The closed half-line $\overline\IR_+=[0,\infty)$ endowed with the action of the multiplicative group $\IR_+$.
\item The one-point compactification $\overline\IR$ of the space $\IR$ of real numbers endowed with the natural action of the additive group $\IR$.
\item The space $\IQ$ of rationals, endowed with the action of the multiplicative group $\IQ^*$ of non-zero rational numbers.
\item The space $\IQ$ of rationals, endowed with the action of the multiplicative group $\IQ_+$ of positive rational numbers.
\item The one-point compactification $\overline\IZ=\IZ\cup\{+\infty\}$ of the discrete space $\IZ$ endowed with the natural action of the additive group $\IZ$ of integer numbers.
\item The one-point compactification of any non-compact locally compact topological group $G$, endowed with the natural action of the topological group $G$.
\end{enumerate}
\end{example}

Given a singular $G$-space $X$, consider the $G$-space $X^\w$ endowed with the Tychonoff product topology and the coordinatewise action of the group $G$. Let $s$ be the unique point of the singleton $\Fix(X;G)$. We shall be interested in two special  subspaces of $X^\w$:
$$X^{<\w}:=\{x\in X^\w:|\{n\in\w:x(n)\ne s\}|<\w\}\mbox{ \ and \ }X^{<\w}_\circ:=X^{<\w}\setminus\{s\}^\w.$$
The orbit space $X^{<\w}_\circ/G$ is called the {\em infinite projective space} of the singular $G$-space $X$ and is denoted by $X\mathsf P^\infty$. 

If $X=\IF$ is a non-discrete topological field endowed with the action of its multiplicative group $\IF^*$, then $\IF^{<\w}$ is a topological vector space over the field $\IF$ and $\IF \mathsf P^\infty$ is the projective space of the topological vector space $\IF^{<\w}$ in the standard sense. In particular, $\IQ\mathsf P^\infty$ is the projective space of the topological vector space $\IQ^{<\w}$ over the topological field $\IQ$ of rational numbers. 

\begin{theorem}\label{p3} The infinite projective space $X\mathsf P^\infty$ of any  singular $G$-space $X$ possesses a canonical superskeleton. If the singular $G$ space $X$ is countable and metrizable, then its infinite projective space $X\mathsf P^\infty$ is  homeomorphic to the space $\IQ\mathsf P^\infty$.
\end{theorem}

\begin{proof} Let $s$ be the unique point of the singleton $\Fix_G(X)$. It will be convenient to identify elements $x\in X^{<\w}$ with sequences $(x_n)_{n\in\w}$.

Since the group $G$ acts by homeomorphisms on the space $X^{<\w}_\circ$, for every open set $U\subseteq X^{<\w}_\circ$ the set $GU=\bigcup_{g\in G}gU$ is open. This implies that the quotient map $q:X^{<\w}_\circ\to X\mathsf P^\infty$ is open.

Now using Proposition~\ref{p2}, we shall prove that $X\mathsf P^\infty$  possesses a canonical superskeleton. For every $n\in\w$ consider  the $G$-invariant subspace 
$$X_n=\{x\in X^{<\w}_\circ:\forall i\in n\;\;x_i=s\}$$
of $X^{<\w}_\circ$. 
Observe that $X_n=\pr^{-1}_n(\{s\}^n)$ where $\pr_n:X^{<\w}_\circ\to X^n$, $\pr_n:x\mapsto (x_0,\dots,x_{n-1})$, is the natural projection.

By Definition~\ref{d:singular}(iii), for every $x\in X\setminus\{s\}$ its orbit $Gx$ is an open set in $X$. Consequently, the singleton $\{s\}=X\setminus\bigcup_{x\in X\setminus\{s\}}Gx$ is closed in $X$ and $\{s\}^n$ is closed in $X^n$, which implies that $X_n$ is closed in $X$. 

Taking into account the openness of the quotient map $q:X^{<\w}_\circ\to Y$ and the $G$-invariantness of the closed sets $X_n$, we conclude that for every $n\in\w$ the set $Y_n=q(X_n)$ is closed in $X\mathsf P^\infty$. Therefire, $(Y_n)_{n\in\w}$ is a vanishing sequence of nonempty closed sets in the space $X\mathsf P^\infty$. Since the singleton $\{s\}$ is nowhere dense in $X$ (by Definition~\ref{d:singular}(iv)), for every $n\in\w$ the set $X_{n+1}$ is nowhere dense in $X_n$ and then the set $Y_{n+1}$ is nowhere dense in $Y_n$.

\begin{claim}\label{cl:sing1} For any point $x\in X^{<\w}_\circ$ its orbit $Gx$ is closed in $X^{<\w}_\circ$.
\end{claim}

\begin{proof} Given any point $y\in X^{<\w}_\circ\setminus Gx$, we should find an open neighborhood $V$ of $y$ in $X^{<\w}_\circ$ such that $V\cap Gx=\emptyset$. Since $y\in X^{<\w}_\circ$, the set $\Omega=\{n\in\w:y_n\ne s\}$ is finite and nonempty. If $x_n\notin Gy_n$ for some $n\in\Omega$, then $Gx_n\cap Gy_n=\emptyset$ and by Definition~\ref{d:singular}(iii),  $Gy_n$ is an open neighborhood of $y_n$ in $X$ and hence 
$V=\{v\in X^{<\w}_\circ:v_n\in Gy_n\}$ is an open neighborhood of $y$ that is disjoint with the orbit $Gx$ of $x$. So, we assume that for every $n\in\Omega$, the point $x_n$ belongs to the orbit $Gy_n$ and hence $x_n=g_ny_n$ for some $g_n\in G$.

If $g_n\ne g_k$ for some numbers $n,k\in\Omega$, then we can choose an open neighborhood $W\subseteq G$ of the identity in the Hausdorff topological group $G$ such that $Wg_n^{-1}\cap Wg^{-1}_k=\emptyset$. By Definition~\ref{d:singular}(iii), the sets $Wy_n$ and $Wy_k$ are open neighborhoods of $y_n$ and $y_k$, respectively. We claim that the open neighborhood $V=\{v\in X^{<\w}_\circ:v_n\in Wy_n,\;v_k\in Wy_k\}$ of $y$ does not intersect the orbit $Gx$ of $x$. In the opposite case we can find an element $g\in G$ such that $gx_n\in Wy_n$ and $gx_k\in Wy_k$. Then  $gg_ny_n=gx_n\in Wy_n$ and $gg_ky_k=gx_k\in Wy_k$. The injectivity of the maps $\alpha_{y_n}$ and $\alpha_{y_k}$ guarantees that $gg_n\in W$ and $gg_k\in W$. Then $g_n\in g^{-1}W\subseteq g_kW^{-1}W$ and finally, $g_nW^{-1}\cap g_kW^{-1}=\emptyset$, which contradicts the choice of the neighborhood $W$. This contradiction shows that $V\cap Gx=\emptyset$.

Finally, assume that $g_n=g_k$ for all $n,k\in \Omega$. Fix any number $n\in\Omega$. Since $y\notin Gx$, there exists $m\in\w\setminus \Omega$ such that $y_m=s\ne g_nx_m$. Then $g_n^{-1}x_m\ne s=y_m$ and by the Hausdorff property of $X$, we can find an open neighborhood $W_m\subseteq X$ of $y_m=s$ such that $g_n^{-1}x_m\notin\overline{W_m}$.
By the continuity of the map $\alpha_{x_m}:G\to X$, $\alpha_{x_m}:g\mapsto gx_m$, the set $F=\alpha_{x_m}^{-1}(\overline{W_m})$ is closed in $G$ and does not contain $g_n^{-1}$. Find an open neighborhood $U\subseteq G$ of the identity such that $F\cap Ug_n^{-1}=\emptyset$. We claim that the open neighborhood $V=\{v\in X^{<\w}_\circ:v_n\in Uy_n,\;v_m\in W_m\}$ of $y$ does not intersect the orbit $Gx$. In the opposite case we can find an element $g\in G$ such that $gx_n\in Uy_n$ and $gx_m\in W_m$. Then $gg_ny_n=g_nx_n\in Uy_n$ and the injectivity of the map $\alpha_{y_n}$ implies that $gg_n\in U$. Also the inclusion $gx_m\in W_m$ implies $g\in F$. Then  $gg_n\in U\cap Fg_n$, which contradicts the choice of the neighborhood $U$.
\end{proof}


\begin{claim}\label{cl:sing2} For every $n\in\w$, the closure $\overline U$ of any nonempty $G$-invariant open set $U\subseteq X_n$ contains some space $X_m$.
\end{claim}

\begin{proof} Fix any point $x\in U$. For every $m>n$, identify the ordinal $m$ with the set $\{0,\dots,m-1\}$ and consider the projection $\pi_m:X_m\to X^{m\setminus n}$, $\pi_m:x\mapsto (x_n,\dots,x_{m-1})$. By the definition of the Tychonoff product topology on $X_n$, there exists $m\ge n$ and an open neighborhood $V\subseteq X^{m\setminus n}$ of $\pi_m(x)$ such that $\pi_m^{-1}(V)\subseteq U$. Then $U=GU\supseteq G\cdot \pi_m^{-1}(V)=\pi^{-1}_m(GV)$. We claim that $\{s\}^{m\setminus n}\subset\overline{GV}$. 
Given any open set $W\subseteq X^{m\setminus n}$ that contains the singleton $\{s\}^{m\setminus n}$, find an open neighborhood $W_s\subseteq X$ of $s$ such that $W_s^{m\setminus n}\subseteq W$. Take any point $z\in X\setminus\{s\}$. By Definition~\ref{d:singular}(v), for every $i\in m\setminus n$, there exists a neighborhood $W_i\subseteq X$ of $s$ such that $\alpha_{x_i}(\alpha_z^{-1}(W_i))\subseteq W_s$. Definition~\ref{d:singular}(iv) ensures that the intersection $Gz\cap\bigcap_{i\in m\setminus n}W_i$ contains some point $w\ne s$. Then the element $g=\alpha_{z}^{-1}(w)$ is well-defined and $gx_i=\alpha_{x_i}(g)\in\alpha_{x_i}(\alpha_z^{-1}(W_i))\subseteq W_s$.
The point $(gx_i)_{i\in m\setminus n}$ belongs to the intersection $W_s^{m\setminus n}\cap \prod_{i\in m\setminus n}gV_i\subseteq W\cap gV\subseteq W\cap GV$, witnessing that $\{s\}^{m\setminus n}\subset \overline{GV}$.

Since the projection $\pi_m:X_n\to X^{m\setminus n}$ is an open $G$-equivariant map,
$$X_m=\pi_m^{-1}(\{s\}^{m\setminus n})\subseteq \pi_m^{-1}(\overline{GV})\subseteq \overline{\pi_m^{-1}(GV)}\subseteq \overline{U}.$$
\end{proof}

\begin{claim}\label{cl:sing3} For any $n\in\w$, point $x\in X\setminus X_n$, and $G$-invariant open neighborhood $U\subseteq X^{<\w}_\circ$ of $x$, there exists a $G$-invariant open neighborhood $V\subseteq U$ of $x$ such that $\overline V\subseteq U\cup X_n$.
\end{claim}

\begin{proof} By the definition of the space $X^{<\w}_\circ\ni x$, there exists an index $i\in\w$ such that $x_i\ne s$. 
 By the definition of the Tychonoff product topology on the regular topological space $X^{<\w}_\circ$, there exist $m>\max\{n,i\}$ and an open set $W\subseteq X^m$ such that $x\in \pr_m^{-1}(W)\subseteq \pr_m^{-1}(\overline{W})\subseteq U$ where $\pr_m:X^{<\w}_\circ\to X^m$, $\pr_m:y\mapsto (y_0,\dots,y_{m-1})$, is the projection onto the first $m$ coordinates. 
By Definition~\ref{d:singular}(v), for every $j\in m$ the point $x_j$ has a neighborhood $W_j\subseteq X$ such that for any neighborhood $O\subseteq X$ of $s$ there exists a neighborhood $O'$ of $s$ such that $\alpha_u(\alpha^{-1}_{x_i}(O'))\subseteq O$ for every $u\in W_j$. Replacing $W_j$ by smaller neighborhoods, we can assume that $\prod_{j\in m}W_j\subseteq W$. 

Now consider the 
``hyperplane'' $H=\{y\in X^m:y_i=x_i\}$ in $X^m$, the open set $V_i=\{(y_j)_{j\in m}\in X^m:y_i\in Gx_i\}$ in $X^m$, and the continuous map $$r:V_i\to H,\;\;r:(y_j)_{j\in m}\mapsto \big((\alpha_{x_i}^{-1}(y_i))^{-1}\cdot y_j\big)_{j\in m}.$$ The map $r$ assigns to each $y\in V_i$ the unique point of the intersection $Gy\cap H$. The continuity of the map $r$ follows from the continuity of the action $\alpha$, the openness and injectivity of the map $\alpha_{x_i}$, and the continuity of the inversion in the topological groups $G$. 
The preimage $V_m:=r^{-1}(\prod_{j\in m}W_j)$ is an open $G$-invariant set in $X^m\setminus\{s\}^m$ and the preimage $V:=\pr_m^{-1}(V_m)$ is an open $G$-invariant neighborhood of $x$ in $X^{<w}_\circ$. We claim that $\overline V\subseteq U\cup X_m\subseteq U\cup X_n$.

Observe that $V_m=r^{-1}(W_m)=GW_m$ where $W_m:=H\cap\prod_{j\in m}W_j$ is an open neighborhood of $\pr_m(x)$ in $H$ and $\overline W_{\!m}\subseteq\overline{W}$.
The continuity of the map $r$ ensures that the set $G\cdot\overline W_{\!m}=r^{-1}(\overline W_{\!m})$ is closed in $V_i$ and hence $\overline V_m\subseteq (X^m\setminus V_i)\cup G\overline W_{\!m}$. 

 We claim that no point of the set $(X^m\setminus V_i)\setminus\{s\}^m$ belongs to the closure $\overline V_m$. Fix any point $y=(y_j)_{i\in m}\in X^m\setminus V_i$ with $y\notin \{s\}^m$. By the definition of the set $V_i$, we have $y_i\notin Gx_i$. If $y_i\ne s$, then 
$U_y=\{z\in X^m:z_i\in Gy_i\}$ is an open neighborhood of $y$, which is disjoint with the set $V_i\supseteq V_m$. 

So, assume that $y_i=s$. Since $y\notin \{s\}^m$, there exists $k\in m$ such that $y_k\ne s$. By the Hausdorff property of $X$, there exists an open neighborhood $O\subseteq X$ of $s$ such that $y_k\notin\overline O$. By the choice of the set  $W_k$, there exists an open neighborhood   $O'\subseteq X$ of $s$ such that $\alpha_u(\alpha_{x_i}^{-1}(O'))\subseteq O$ for every $u\in W_k$. Now consider the open neighborhood $U_y=\{z\in X^m:y_i\in O',\;y_k\notin \overline O\}$ of $y$ in $X^m$. We claim that $U_y\cap V_m=\emptyset$. To derive a contradiction, assume that $U_y\cap V_m$ contains some point $z$. Since $z\in V_m=GW_m$, the point $z$ can be written as  $z=gh$ for some $g\in G$ and  $h\in W_m=H\cap \prod_{j\in m}W_j$. Write $h$ as $(h_j)_{j\in m}$. It follows from $h\in H$ that $h_i=x_i$ and $h_k\in W_k$. On the other hand, $z\in U_y$ implies $gx_i=gh_i\in O'$ and $gh_k\notin \overline O$. The inclusions $gx_i\in O'$ and $h_k\in W_k$ imply $g\in \alpha_{x_i}^{-1}(O')$ and $gh_k=\alpha_{h_k}(g)\in \alpha_{h_k}(\alpha_{x_i}^{-1}(O')\subseteq O$, which contradicts $gh_k\notin \overline O$.
This contradiction shows that $\overline V_m\subseteq G\overline W_m\cup\{s\}^m$.
Since the projection $\pr_m:X^{<\w}_\circ \to X^m$ is open, 
\begin{multline*}
\overline V=\pr_m^{-1}(\overline V_{\!m})\subseteq \pr_m^{-1}(\{s\}^m\cup G{\cdot}\overline W_{\!m})=X_m\cup G{\cdot}\pr_m^{-1}(\overline W_{\!m})\subseteq\\
 X_m\cup G{\cdot} \pr_m^{-1}(W)\subseteq X_m\cup G{\cdot} U= X_m\cup U.
\end{multline*} 
\end{proof}

Claims~\ref{cl:sing1}--\ref{cl:sing3} imply that the $G$-space $X^{<\w}_\circ$ satisfies the conditions of Proposition~\ref{p2}. The proof of this proposition implies that the sequence $(Y_n)_{n\in\w}$ is a superskeleton form the orbit space $X\mathsf P^\infty=X^{<\w}_\circ/G$. Since each space $Y_{n+1}$ is nowhere dense in $Y_n$, the superskeleton $(Y_n)_{n\in\w}$ is caninical. If the space $X$ is countable and first-countable, then $X\mathsf P^\infty$ is homeomorphic to $\IQ\mathsf P^\infty$ by Theorem~\ref{t:main}.
\end{proof}

\begin{remark} The skeleton $(Y_n)_{n\in\w}$ contructed in the proof of Theorem~\ref{p3} will be called {\em the canonical superskeleton} of the space $X\mathsf P^\infty$.
\end{remark}


Let $\IF$ be a topological field. Three elements $\IF^*x,\IF^*y,\IF^*z$ of the projective space $\IF\mathsf P^\infty$ are called {\em collinear} if the union $\IF^*x\cup\IF^*y\cup\IF^* z$ is contained in some $2$-dimensional vector subspace of $\IF^{<\w}$.

For two topological fileds $\IF_1,\IF_2$ a map $f:\IF_1\mathsf P^\infty\to\mathsf \IF_2\mathsf P^\infty$ is called {\em affine} if for any collinear elements $\IF_1^*x,\IF_1^*y,\IF_1^*z\in\IF_1\mathsf P^\infty$, the elements $f(\IF_1^*x),f(\IF_1^*y),f(\IF_1^*z)$ are collinear in the projective space $\IF_2^*\mathsf P^\infty$. A bijective map  $f:\IF_1\mathsf P^\infty\to\mathsf \IF_2\mathsf P^\infty$ is called an {\em affine isomorphism} if both maps $f$ and $f^{-1}$ are affine. 
If an affine isomorphism  $f:\IF_1\mathsf P^\infty\to\mathsf \IF_2\mathsf P^\infty$ is also a homeomorphism, then $f$ is called an {\em affine topological isomorphism}. The projective spaces $\IF_1\mathsf P^\infty,\IF_2\mathsf P^\infty$ are called {\em affinely isomorphic} (resp. {\em affinely homeomorphic}) if there exists an affine topological ismorphism $f:\IF_1\mathsf P^\infty\to\mathsf \IF_2\mathsf P^\infty$.

  In spite of the fact that for any countable subfields $\IF_1,\IF_2\subseteq\mathbb C$, the infinite projective spaces $\IF_1\mathsf P^\infty$ and $\IF_2\mathsf P^\infty$ are homeomorphic (by Theorem~\ref{p3}),  we have the following rigidity result for affine isomorphisms between infinite projective spaces.

\begin{proposition} Two (topological) fields $\IF_1,\IF_2$ are (topologically) isomorphic if and only if their infinite projective spaces $\IF_1\mathsf P^\infty,\;\IF_2\mathsf P^\infty$ are affinely isomorphic (affinely homeomorphic).
\end{proposition}

\begin{proof} If $\sigma:\IF_1\to\IF_2$ is a (topological) isomorphism of the (topological) fields $\IF_1,\IF_2$, then the map $f$ induces the affine (topological) isomorphism  $$\tilde \sigma:\IF_1\mathsf P^\infty\to\IF_2\mathsf P^\infty,\;\;\tilde \sigma:\IF_1^*\cdot (x_n)_{n\in\w}\mapsto \IF_2^*\cdot (\sigma(x_n))_{n\in\w},$$ of the projective spaces $\IF_1\mathsf P^\infty$ and $\IF_2\mathsf P^\infty$. This proves the ``only if'' part of the proposition.
\smallskip

To prove the ``if'' part, assume that $g:\IF_1\mathsf P^\infty\to\IF_2\mathsf P^\infty$ is an affine (topological) isomorphism between the projective spaces $\IF_1\mathsf P^\infty$ and $\IF_2\mathsf P^\infty$. Identify the space $\IF_1^3$ with the $3$-dimensional vector subspace $\{(x_n)_{n\in\w}\in\IF_1^{<\w}:\forall n\ge 3\;(x_n=0)\}$ of the topological vector space $\IF_1^{<\w}$. Consider the vectors $e_1=(1,0,0)$, $e_2=(0,1,0)$ and $e_3=(0,0,1)$ in $\IF_1^3$, and observe that the elements $\IF_1^*e_1,\IF_1^*e_2,\IF_1^*e_3$ are not collinear in the projective space $\IF_1\mathsf P^\infty$. Then their images $g(\IF_1^*e_1),g(\IF_1^*e_2),g(\IF_1^*e_3)$ are not collinear in the projective space $\IF_2\mathsf P^\infty$ and hence the union $g(\IF_1^*e_1)\cup g(\IF_1^*e_2) \cup g(\IF_1^*e_3)$ is contained in a unique 3-dimensional vector subspace $L$ of the topological vector space $\IF_2^{<\w}$. We can choose a basis $e_1'\in g(\IF^*_1e_1)$, $e_2'=g(\IF^*_1e_2)$, $e_3'\in g(\IF^*_1e_3)$ for the space $L$ such that $\IF_2^*(e'_1+e'_2+e'_3)=g(\IF^*_1(e_1+e_2+e_3))$.
Then the affine (topological) isomorphism $g$ induces an affine (topological) isomorphism $h=g{\restriction}\IF_1\mathsf P^2$ of the projective planes $\IF_1 P^2=(\IF^3_1\setminus\{0\}^3)/\IF_1^*$ and  $\IF_2 P^2=(\IF^3_2\setminus\{0\}^3)/\IF_2^*$ such that $h(\IF_1^*e_i)=\IF_2^*e_i '$ for $i\in\{1,2,3\}$ and $h(\IF_1^*(e_1+e_2+e_3))=\IF^*_2(e_1'+e_2'+e_3')$. By a classical result of Hilbert (cf. Proposition 3.11 in \cite{Hart} or  Lemma 2.8.2 in \cite{Kryftis}), there exists an isomorphism $\sigma:\IF_1\to\IF_2$ of the fields $\IF_1,\IF_2$ such that for any $(x_1,x_2,x_3)\in\IF_1^3$ we have  $$h(\IF_1^*(x_1e_1+x_2e_2+x_3e_3))=\IF_2^*(\sigma(x_1)e_1'+\sigma(x_2)e_2'+
\sigma(x_3)e_3').$$ Therefore, the fields $\IF_1,\IF_2$ are isomorphic.

If the affine isomorphism $g$ is a homeomorphism, then so is the map $h$. In this case we shall prove that the field isomorphism $\sigma$ is a homeomorphism. For this observe that for every $i\in\{1,2\}$, the map $f_i:\IF_i\to \IF_i\mathsf P^\infty$, $f:x\mapsto \IF_i^*(e_1+xe_2)$ is a topological embedding. Since $\sigma=f_2^{-1}\circ h\circ f_1$ and $\sigma^{-1}=f_1^{-1}\circ h^{-1}\circ f_2$, then maps $\sigma$, $\sigma^{-1}$ are continuous and $\sigma:\IF_1\to\IF_2$ is a topological isomorphism of the topological fields $\IF_1,\IF_2$.
\end{proof}



By Example~\ref{ex1}(2,5), the spaces $\mathbb C$, $\mathbb R$, $\bar\IR_+$ endowed with suitable group actions are singular $G$-spaces. By Theorem~\ref{p3}, the infinite projective spaces $\mathbb C\mathsf P^\infty$, $\IR\mathsf P^\infty$, $\overline\IR_+\mathsf P^\infty$ possess (canonical) superskeleta. It can be shown that each of these spaces has a countable base of the topology  consisting of sets, homeomorphic to the space $\IR^{<\w}$, so is a (non-metrizable) $\IR^{<\w}$-manifold. It can be shown that the $\IR^{<\w}$-manifolds  $\mathbb C\mathsf P^\infty$, $\IR\mathsf P^\infty$, $\overline\IR_+\mathsf P^\infty$  are pairwise non-homeomorphic (because of different homotopical properties of complements $Y_0\setminus Y_n$ of their canonical skeleta). The distinguishing topological property of the space $\overline\IR_+\mathsf P^\infty$ is possessing a superskeleton $(Y_n)_{n\in\w}$ such that for every $n<m$ in $\w$ the complement $Y_n\setminus Y_m$ is contractible.

This observation and the topological characterization of the space $\IQ\mathsf P^\infty$ suggests the following topological characterization of the space $\overline\IR_+\mathsf P^\infty$.

\begin{conjecture} A Hausdorff topological space $X$ is homeomorphic to $\overline\IR_+\mathsf P^\infty$ if and only if $X$ possesses a superskeleton $(X_n)_{n\in\w}$ such that  for every $n<m$ in $\w$ the set $X_{m}$ is a $Z$-set in $X_n$ and the space $X_n\setminus X_m$ is homeomorphic to $\IR^{<\w}$.
\end{conjecture}

A closed subset $A$ of a topological space $X$ is called a {\em $Z$-set} in $X$ if the set $C(\mathbb I^\w,X\setminus A)$ is dense in the space $C(\mathbb I^\w,X)$ of continuous functions from the Hilbert cube $\mathbb I^\w=[0,1]^\w$ to $X$, endowed with the compact-open topology. For more information on Infinite-Dimensional Topology, see the monographs \cite{BP}, \cite{vM}, \cite{BRZ}, \cite{Sakai}. For the topological characterization of the space $\IR^{<\w}$, see \cite{Mog}, \cite{CDM}, \cite[\S1.6]{BRZ}, \cite[\S4.3]{Sakai}.

It can be shown that the spaces $\IR\mathsf P^\infty$,  $\IC\mathsf P^\infty$,  $\overline \IR_+\mathsf P^\infty$ contain dense subspaces, homeomorphic to $\IQ\mathsf P^\infty$.

\begin{problem} Does the Golomb (or Kirch) space contain a subspace homeomorphic to $\IQ\mathsf P^\infty$?
\end{problem}

\section{Some properties of skeleta in topological spaces}

In this section we establish some properties of various skeleta in topological spaces. First we fix some standard notations.

For a subset $A$ of a topological space $X$ by $\bar A$ and $\partial A$ we denote the closure and boundary of $A$ in $X$. By $\w$ we denote the smallest infinite ordinal, and by $\IN$ the set $\w\setminus\{0\}$ of positive integer numbers. Ordinals are identified with the sets of smaller ordinals. So, $n=\{0,\dots,n-1\}$ for any natural number $n\in\w$.

\begin{lemma}\label{l:cr} 
A (Hausdorff second-countable) topological space  is coregular if (and only if) it has a coregular skeleton.
\end{lemma}

\begin{proof} To prove the ``if'' part, assume that a topological space $X$ has a coregular skeleton $(X_n)_{n\in\w}$. First we show that the space $X$ is Hausdorff. Fix any distinct points $x,y\in X$. By Definition~\ref{d:skeleton}, $(X_n)_{n\in\w}$ is a vanishing sequence of closed sets in $X$. Consequently, $\bigcap_{n\in\w}X_n$ and we can find a number $m\in\w$ such that $x,y\in X\setminus X_m$. Since $(X_n)_{n\in\w}$ is a coregular skeleton, the space $X\setminus X_m$ is regular and hence Hausdorff. Then the points $x,y$ have disjoint open neighborhoods $O_x,O_y$ in the space $X\setminus X_m$. Since $X\setminus X_m$ is open in $X$, the sets $O_x,O_y$ remain open in $X$, witnessing that $X$ is Hausdorff.

Now fix any non-empty open sets $U_1,\dots,U_k\subseteq X$. By Definition~\ref{d:skeleton}, for every $i\in\{1,\dots,k\}$ there exists a number $n_i\in\w$ such that $X_{n_i}\subseteq\overline U_{\!i}$. Then for the number $n=\max_{i\le k}n_i$ we have $$X_n\subseteq\bigcap_{i=1}^nX_{n_i}\subseteq\bigcap_{i=1}^k\overline U_{\!i}.$$
Since the skeleton $(X_i)_{i\in\w}$ is coregular, the space $X\setminus X_n$ is regular and so is its subspace $X\setminus(\overline U_{\!1}\cap\dots\cap\overline U_{\!k})$. This completes the proof of the coregularity of $X$.
\smallskip

To prove the ``only if'' part, assume that the space $X$ Hausdorff, second-countable, and coregular. If $X$ is regular, then put $X_0=X$, $X_n=\emptyset$ for all $n\in\IN$, and observe that $(X_n)_{n\in\w}$ is a coregular skeleton for $X$. So, we assume that $X$ is not regular and hence infinite.

Fix a countable base $\{U_n\}_{n\in\IN}$ of the topology of $X$ such that $U_n\ne\emptyset$ for all $n\in\IN$. Let $X_0=X$ and $X_n=\overline U_{\!1}\cap\cdots\cap\overline U_{\!n}$ for every $n\in\IN$. 

To show that $\bigcap_{n\in\w}X_n=\emptyset$, fix any point $x\in X$. Since the space $X$ is infinite and Hausdorff, there exists $n\in\w$ such that $x\notin\overline U_{\!n}$ and hence $x\notin X_n$. This shows that the sequence $(X_n)_{n\in\w}$ is vanishing. By the coregularity of $X$, for every $n\in\w$ the space $X\setminus X_n$ is regular. Also observe that every nonempty open set $U\subseteq X$ contains some set $U_n$ and then $X_n\subseteq \overline U_{\!n}\subseteq \overline U$. Therefore, the vanishing sequence $(X_n)_{n\in\w}$ is a coregular skeleton for the space $X$.
\end{proof}



\begin{lemma}\label{l:sc} A (Hausdorff infinite second-countable) topological space  is superconnected if (and only if) it possesses a superconnecting skeleton.
\end{lemma}

\begin{proof} To prove the ``if'' part, assume that a topological space $X$ has a superconnecting skeleton $(X_n)_{n\in\w}$. By Definition~\ref{d:skeleton}, $(X_n)_{n\in\w}$ is a vanishing sequence in $X$. To see that $X$ is superconnected, fix any non-empty open sets $U_1,\dots,U_k\subseteq X$. By Definition~\ref{d:skeleton}, for every $i\in\{1,\dots,k\}$ there exists a number $n_i\in\w$ such that $\emptyset\ne X_{n_i}\subseteq\overline U_{\!i}$. Then for the number $n=\max_{i\le k}n_i$ we have $$\emptyset \ne X_n\subseteq\bigcap_{i=1}^nX_{n_i}\subseteq\bigcap_{i=1}^k\overline U_{\!i},$$
witnessing that $\overline U_{\!1}\cap\cdots\cap\overline U_{\!k}\ne\emptyset$ and $X$ is superconnected.
\smallskip

To prove the ``only if'' part, assume that an infinite Hausdorff second-countable space $X$ is superconnected and fix a countable base $\{U_n\}_{n\in\IN}$ of the topology of $X$ such that $U_n\ne\emptyset$ for all $n\in\IN$. Let $X_0=X$.
By the superconnectedness of $X$, for every $n\in\IN$ the closed set $X_n=\overline U_{\!1}\cap\cdots\cap\overline U_{\!n}$ is not empty. 

To show that $\bigcap_{n\in\w}X_n=\emptyset$, fix any point $x\in X$. Since the space $X$ is infinite and Hausdorff, there exists $n\in\w$ such that $x\notin\overline U_{\!n}$ and hence $x\notin X_n$. This shows that the sequence $(X_n)_{n\in\w}$ is vanishing.

To see that $(X_n)_{n\in\w}$ is a superconnecting skeleton for $X$, take any nonempty open set $U\subseteq X$ and find $n\in\w$ such that $U_n\subseteq U$. Then $$\emptyset\ne X_n=\overline U_{\!1}\cap\dots\cap\overline U_{\!n}\subseteq\overline U_{\!n}\subseteq\overline U.$$
\end{proof}



\begin{lemma}\label{l:crowded}  Let $X$ be an infinite Hausdorff topological space and $(X_n)_{n\in\w}$ be its superconnecting skeleton. Then the space $X$ is crowded and for every $n\in\w$, the space $X_n$ is infinite.
\end{lemma}

\begin{proof} First we show that for every $n\in\w$ the set $X_n$ is infinite. To derive a contradiction, assume that for some $n\in\w$ the space $X_n$ is finite. We can assume that $n$ is the smallest number with this property. Then for every $i\in n$ the space $X_i$ is infinite and hence contains some point $x_i$.  Since $X$ is infinite and Hausdorff, there exists a nonempty set $U\subseteq X$ whose closure $\overline U$ does not intersect the finite set $X_n\cup\{x_i\}_{i\in n}$. Since the skeleton $(X_k)_{k\in\w}$ is superconnecting, the closure $\overline{U}$ contains some set $X_i\ne\emptyset$. Assuming that $i\ge n$ we conclude $\emptyset\ne X_i=X_i\cap\overline U\subseteq X_n\cap\overline U=\emptyset$, which is a contradiction. So, $i<n$ and then $x_i\in X_i\subseteq\overline U$, which contradicts the choice of $U$. This contradiction witnesses that the spaces $X_n$ are infinite. 

Assuming that the space $X$ is not crowded, we can find an isolated point $x$ in $X$. Then $U=\{x\}$ is an open subspace of $X$ such that  $\overline{U}=\{x\}$ by the Hausdorff property of $X$. Since $(X_n)_{n\in\w}$ is a superconnecting skeleton of $X$, there exist $n\in\w$ such that $X_n\subseteq \overline{U}=\{x\}$ and hence $X_n$ is finite, which is a desired contradiction.
\end{proof}

\begin{lemma}\label{l:supercrowded}  Let $X$ be an infinite topological space and $(X_n)_{n\in\w}$ be its superskeleton. Then for every $n\in\w$ the space $X_n$ is crowded.
\end{lemma}

\begin{proof} By Definition~\ref{d:skeleton}, for every $n\in\w$ the sequence $(X_m)_{m=n}^\infty$ is a superconnecting coregular skeleton for the space $X_n$. Then the space $X=X_0$ is coregular and hence Hausdorff. By Lemma~\ref{l:crowded}, for every $n\in\w$ the set $X_n$ is infinite. Applying Lemma~\ref{l:crowded} to the superconnecting skeleton $(X_m)_{m=n}^\infty$ for the infinite Hausdorff space $X_n$, we conclude that $X_n$ croweded.
\end{proof}

\begin{lemma}\label{l:nodense} Let $X$ be an infinite Hausdorff space and $(X_n)_{n\in\w}$ be its superconnecting or coregular skeleton. Then for some $n\in\w$ the set $X_n$ is nowhere dense in $X=X_0$.
\end{lemma}

\begin{proof}  Being infinite and Hausdorff, the space $X$ contains two disjoint  nonempty open sets $U,V$. By Definition~\ref{d:skeleton}, the closures $\overline U$, $\overline V$ contain some set $X_n$. Then  the set $X_n\subseteq\overline U\cap\overline V\subseteq \overline U\cap\overline{X\setminus U}=\partial U$ is nowhere dense in $X$.
 \end{proof}
 
 \begin{lemma}\label{l:canonical} If $(X_n)_{n\in\w}$ is a superskeleton for an infinite topological space $X$, then for some increasing number sequence $(n_k)_{k\in\w}$ the sequence $(X_{n_k})_{k\in\w}$ is a canonical superskeleton for $X$.
 \end{lemma}
 
 \begin{proof} Applying Lemma~\ref{l:nodense}, construct inductively an increasing number sequence $(n_k)_{n\in\w}$ such that $n_0=0$ and for every $k\in\w$ the set $X_{n_{n+1}}$ is nowhere dense in $X_{n_k}$. 
 \end{proof}
 
 A subset of a topological space  is called {\em regular open} if it coincides with the interior of its closure. A topological space $X$ is called {\em semiregular} if it is Hausdorff and has a base consisting of regular open sets.

\begin{lemma}\label{l:corsem} Each coregular topological space $X$ is semiregular.
\end{lemma}

\begin{proof} If the space $X$ is finite, then it is discrete (being Hausdorff) and hence regular and semiregular. So, we assume that $X$ is infinite. 

To show that $X$ is semiregular, fix any point $x\in X$ and an open  neigborhood $U$ of $x$ in $X$. Taking into account that $X$ is infinite and Hausdorff, we can replace $U$ by a smaller neighborhood of $x$ and assume that $X$ contains a non-empty open set $W$, which is disjoint with $U$. Then $U\cap \overline W=\emptyset$. Since $X$ is coregular, the space $X\setminus \overline W$ is regular. Then the point $x$ has an open neighborhood $V\subseteq X\setminus \overline W$ such that $V\subseteq U$ and $\overline V\cap (X\setminus \overline W)\subseteq U$. Let $O$ be the interior of the set $\overline V$ in $X$.
Observe that $O\cap W\subseteq \overline V\cap W\subseteq\overline U\cap W=\emptyset$ and hence $O\cap \overline W=\emptyset$. Then $O\subseteq\overline V\setminus \overline W\subseteq U$. Taking into account that the set $O$ is regular open, we conclude that the space $X$ is semiregular.
\end{proof}

\begin{lemma}\label{l:product} A topological space $X$ is regular if and only if its square $X\times X$
is coregular.
\end{lemma}

\begin{proof} The ``only if'' part is trivial. To prove the ``if'' part, assume that the space $X\times X$ is coregular.   Then $X\times X$ is Hausdorff and so is the space $X$. If $X$ is finite, then $X$ is discrete and hence regular. So, assume that $X$ is infinite. Then we can fix any point $x\in X$ and find a non-empty open set $U\subseteq X$ such that $x\notin \overline{U}$. By the coregularity of $X\times X$ the compement $X\times X\setminus\overline{U\times U}$ is a regular space and so is its subspace $\{x\}\times X$ and the space $X$.
\end{proof} 

\section{Main Results}\label{s:main}

In this section we prove a difficult Theorem~\ref{t:key} implying Theorem~\ref{t:main}  and many other important properties of the space $\IQ\mathsf P^\infty$. Let us recall that a function $f:X\to Y$ between topological spaces $X,Y$ is called a {\em topological embedding} if $f$ is a homeomorphism between $X$ and the subspace $f(X)$ of $Y$.

\begin{theorem}\label{t:key} Let $X$ be a countable second-countable space, $(X_n)_{n\in\w}$ be a coregular skeleton in $X$, and $A$ be a nowhere dense closed  set in $X$. Let $Y$ be a countable second-countable space, $(Y_n)_{n\in\w}$ be a canonical superskeleton in $Y$, and $B$ be a subset of $Y$ such that for every $n\in\w$ the intersection $\bar B\cap Y_n$ is nowhere dense in $Y_n$. Let $f:A\to B$ be a homeomorphism such that $f^{-1}(Y_n)=A\cap X_n$ for all $n\in\w$. Then there exist 
a topological embedding $\bar f:X\to Y$ such that $\bar f{\restriction}A=f$ and $\bar f^{-1}(Y_n)=X_n$ for all $n\in\w$. If the sequence $(X_n)_{n\in\w}$ is a canonical superskeleton in $X$, the set $B$ is closed in $Y$, and for every $n\in\w$ the set $A\cap X_n$ is nowhere dense in $X_n$, then $\bar f(X)=Y$ and $\bar f$ is a homeomorphism.
\end{theorem}

\begin{proof} For constructing the topological embedding $\bar f:X\to Y$ we should make some preliminary work with the spaces $X$ and $Y$.

We start with the space $X$ endowed with a coregular skeleton $(X_n)_{n\in\w}$. 
Let $\ell_X:X\to\w$ be the function assigning to each $x\in X$ the largest number $n$ such that $x\in X_n$ (such the number $n$ exists because $\bigcap_{i\in\w}X_i=\emptyset$).  The function $\ell_X$ will be called the {\em  level map} of $X$. 

Denote by $\tau_X$ the topology of $X$.  Using the first-countability of $X$, for each point $x\in X$, fix a neighborhood base $\{O^X_n(x)\}_{n\in\w}$ at $x$ such that $O^X_{n+1}(x)\subseteq O^X_n(x)\subseteq X\setminus X_{1+\ell_X(x)}$ for every $n\in\w$.  Fix a well-order $\preceq_X$ on the set $X':=X\setminus A$ such that for any $x\in X'$ the set ${\downarrow}x:=\{z\in X':z\preceq_X x\}$ is finite. For a nonempty subset $S\subseteq X'$ by $\min S$ we denote the smallest element of the set $S$ with respect to the well-order $\preceq_X$.

Next, do the same for the space $Y$ endowed with a canonical superskeleton $(Y_n)_{n\in\w}$. Denote by $\tau_Y$ the topology of $Y$.
Let  $\ell_Y:Y\to\w$ be the function assigning to each $y\in Y$ the largest number $n$ such that $x\in Y_n$. 
Using the first countability of $Y$, for every point $y\in Y$ choose a neighborhood base $\{O^Y_n(y)\}_{n\in\w}$ at $y$ such that $O^Y_{n+1}(y)\subseteq O_n^Y(y)\subseteq Y\setminus Y_{1+\ell_Y(y)}$ for all $n\in\w$. Fix a well-order $\preceq_Y$ on the set $Y'=Y\setminus \overline B$ such that for any $y\in Y$ the set ${\downarrow}y:=\{z\in Y:z\preceq_Y y\}$ is finite.
\smallskip

If the set $X'=X\setminus A$ is finite, then let $\bar f:X\to Y$ be any injective function such that $\bar f{\restriction}A=f$ and $f(x)\in \ell_Y^{-1}(\ell_X(x))\setminus\bar B$ for any $x\in X'$. The choice of $\bar f$ is possible since for every $n\in\w$ the sets $\bar B\cap Y_n$ and $Y_{n+1}$ are nowhere dense in $Y_n$. Then $\bar f$ is a required extension of $f$. 

So, we assume that the open subspace $X'=X\setminus A$ of $X$ is infinite. In this case we shall construct the topological embedding $\bar f$  by induction over the index set $\Gamma=\w\cup(\w\times\w)$ endowed with the strict well-order $\prec$ uniquely defined by the following conditions:
\begin{itemize}
\item[a)] for numbers $n,m\in\w$ we have $n\prec m$ iff $n< m$;
\item[b)] for a number $n\in\w$ and a pair $(i,j)\in\w\times\w$ we have $(i,j)\prec n$ iff $i+j<n$; 
\item[c)] for a pair $(i,j)\in\w\times\w$ and a number $n\in\w$ we have $n\prec (i,j)$ iff $n\le i+j$;
\item[d)] for two pairs $(i,j),(n,m)\in\w\times \w$ we have $(i,j)\prec(n,m)$ iff either $i+j<n+m$ or $i+j=n+m$ and $i<n$.
\end{itemize}
The initial elements of the well-ordered set $\Gamma$ are:
$$0,(0,0),1,(1,0),(0,1),2,(2,0),(1,1),(0,2),3,(3,0),(2,1),(1,2),(0,3),4,\dots $$
For every element $\gamma\in \Gamma$ let ${\downarrow}\gamma=\{\alpha\in\Gamma:\alpha\prec \gamma\}$. Writing $k\in{\downarrow}\gamma$ (resp. $(i,j)\in{\downarrow}\gamma$) we shall understand that $k\in\w\cap{\downarrow}\gamma$ (resp. $(i,j)\in(\w\times\w)\cap{\downarrow}\gamma$).
\smallskip

Write the set $\w$ as the union $\Omega\cup\overrightarrow\Omega\cup\overleftarrow\Omega$ of pairwise disjoint sets $\Omega,\overrightarrow\Omega,\overleftarrow\Omega$ such that $|\Omega|=|A|$, $|\overrightarrow\Omega|=|X'|=\w$, $0\in\overrightarrow\Omega$, and $|\overleftarrow \Omega|\in\{0,\w\}$. We choose the set $\overleftarrow\Omega$ to be infinite iff  the skeleton $(X_n)_{n\in\w}$ is a canonical superskeleton for $X$, the set $B$ is closed in $Y$, and for every $n\in\w$ the set $A\cap X_n$ is nowhere dense in $X_n$. Let $\xi:\Omega\to A$ be a bijective function.

Now we are ready to start the inductive construction of the topological embedding $\bar f:X\to Y$ extending the homeomorphism $f$.

Inductively we shall construct sequences of points $\{x_n\}_{n\in\w}\subseteq X$, $\{y_n\}_{n\in\w}\subseteq Y$, a double sequences of open sets $\{U_{n,k}\}_{n,k\in\w}\subseteq\tau_X$, $\{V_{n,k}\}_{k,n\in\w}\subseteq\tau_Y$, and a function $\ell:\Gamma\to\w$ 
such that for any $\gamma\in \Gamma$ the following conditions are satisfied:
\begin{enumerate}
\item If $\gamma=n$ for some number $n\in\w$, then 
\begin{itemize}
\item[(1a)] $\ell(\gamma)=\ell_X(x_n)=\ell_Y(y_n)$;
\item[(1b)] $x_n\notin\{x_k\}_{k\in{\downarrow}\gamma}$ and  $y_n\notin\{y_k\}_{k\in{\downarrow}\gamma}$;
\item[(1c)] $\{(i,j)\in{\downarrow}\gamma:x_n\in U_{i,j}\}=\{(i,j)\in{\downarrow}\gamma:y_n\in V_{i,j}\}$;
\item[(1d)] $\{(i,j)\in{\downarrow}\gamma:x_n\in \overline U_{\!i,j}\}=\{(i,j)\in{\downarrow}\gamma:y_n\in \overline V_{\!i,j}\}$;
\item[(1e)] If $n\in\Omega$, then $x_n=\xi(n)$ and $y_n=f(x_n)$;
\item[(1f)] If $n\in\overrightarrow\Omega$, then $x_n=\min (X'\setminus \{x_k\}_{k\in{\downarrow}\gamma})$ and $y_n\notin \overline B$;
\item[(1g)] If $n\in\overleftarrow\Omega$, then $y_n=\min (Y'\setminus \{y_k\}_{k\in{\downarrow}\gamma})$ and $x_n\notin A$.
\end{itemize}
\item If $\gamma=(n,k)$ for some $n,k\in\w$, then 
\begin{itemize}
\item[2a)] $\ell(\gamma)\ge 2+\max\{\ell(\alpha):\alpha\in{\downarrow}\gamma\}$;
\item[2b)] for any $m\in\w\cap{\downarrow}\gamma$ with $m\ne n$, we have $x_m\notin \overline{U}_{\!n,k}$ and $y_m\notin\overline{V}_{\!n,k}$;
\item[2c)] $x_n\in U_{n,k}\subseteq O^X_{k}(x_n)\subseteq X\setminus X_{1+\ell(n)}$ and $y_n\in V_{n,k}\subseteq O^Y_{k}(x_n)\subseteq Y\setminus Y_{1+\ell(n)}$;
\item[2d)] $\{(i,j)\in{\downarrow}\gamma:U_{n,k}\subseteq U_{i,j}\}= \{(i,j)\in{\downarrow}\gamma:x_{n}\in U_{i,j}\}$ and \newline$\{(i,j)\in{\downarrow}\gamma:V_{n,k}\subseteq V_{i,j}\}= \{(i,j)\in{\downarrow}\gamma:y_n\in V_{i,j}\}$;
\item[2e)] $\{(i,j)\in{\downarrow}\gamma:U_{n,k}\cap \overline U_{\!i,j}=\emptyset\}= \{(i,j)\in{\downarrow}\gamma:x_{n}\notin\overline U_{\!i,j}\}$ and\newline $\{(i,j)\in{\downarrow}\gamma:V_{n,k}\cap\overline V_{\!i,j}=\emptyset\}= \{(i,j)\in{\downarrow}\gamma:y_n\notin \overline V_{\!i,j}\}$;
\item[2f)] $X_{\ell(\gamma)}=\partial U_{n,k}$ and $Y_{\ell(\gamma)}=\partial V_{n,k}\subseteq\overline{V_{n,k}\cap Y_{\ell(n)}}$;
\item[2g)] if $n\in\Omega$, then $f(U_{n,k}\cap A)=V_{n,k}\cap B$;
\item[2h)] If $n\notin\Omega$, then $U_{\!n,k}\cap A=\emptyset=V_{n,k}\cap \bar B$;
\item[2i)] If $\overleftarrow\Omega\ne\emptyset$, then $X_{\ell(\gamma)}=\partial U_{n,k}\subseteq\overline U_{\!n,k}\cap X_{\ell(n)}$.
\end{itemize}
\end{enumerate}
\smallskip

0. We start the inductive construction letting $x_0$ be the smallest point of the well-ordered set $(X',\preceq_X)$. Since the set $Y_{1+\ell_X(x_0)}$ is nowhere dense in $Y_{\ell_X(x_0)}$, the set $Y_{\ell_X(x_0)}\setminus Y_{1+\ell_X(x_0)}$ is not empty and hence contains some point $y_0$. 
Such choice of $y_0$ guarantees that the condition (1) is satisfied for $\gamma=0\in\overrightarrow\Omega$.
\smallskip

Now assume that for some $\gamma\in\Gamma$, we have defined the function $\ell$  on the set ${\downarrow}\gamma$ and constructed points $x_n,y_n$ and  open sets $U_{i,j},V_{i,j}$ for all $n\in\w\cap{\downarrow}\gamma$ and $(i,j)\in(\w\times\w)\cap{\downarrow}\gamma$ so that the inductive conditions (1)--(2) are satisfied. 

To fulfill the inductive step, consider two possible cases.
\smallskip

1. First assume that $\gamma=n$ for some number $n\in\w$. This case has three subcases.

$1'$. If $n\in\Omega$, then put $x_n=\xi(n)\in A$ and $y_n=f(x_n)\in B$. Our assumption on the map $f$ ensures that $\ell_X(x_n)=\ell_Y(y_n)$. So we can put $\ell(\gamma):=\ell_X(x_n)=\ell_Y(y_n)$ and see that the inductive conditions (1a), (1b) are satisfied. To see that (1c) is satisfied, take any pair $(i,j)\prec \gamma=n$.

First we assume that $x_n\in U_{i,j}$.  If $x_i\notin A$, then $i\notin \Omega$ and we obtain a contradiction $x_n\in A\cap U_{i,j}=\emptyset$ applying the inductive condition (2h). This contradiction shows that $x_i\in A$. Then $y_n=f(x_n)\in f(A\cap U_{i,j})=B\cap V_{i,j}\subseteq V_{i,j}$ by the inductive condition (2g). By analogy we can show that $y_n\in V_{i,j}$ implies $x_n\in U_{i,j}$. This means that the condition (1c) is satisfied. 

Now assume that $x_n\in\overline U_{\!i,j}$. If $x_n\in U_{i,j}$, then $y_n\in V_{i,j}\subseteq\overline V_{i,j}$ by the (already proved) inductive condition (1c). So, we assume that $x_n\in\overline U_{\!i,j}\setminus U_{i,j}=\partial U_{i,j}=X_{\ell(i,j)}$ (for the last equality, see the inductive condition (2f)). Then $\ell_Y(y_n)=\ell_X(x_n)\ge\ell(i,j)$ and hence $y_n\in Y_{\ell(i,j)}\subseteq \overline V_{\!i,j}$ by the condition (2f). By analogy we can prove that $y_n\in \overline V_{\!i,j}$ implies $x\in\overline U_{\!i,j}$. This means that the condition (1d) is satisfied. It is clear that the conditions (1e)--(1g) are satisfied, too.
\smallskip

$1''$. Next, assume that $n\in\overrightarrow\Omega$. This case requires much more work. Define the point $x_n$ by the formula (1f) and put $\ell(n)=\ell_X(x_n)$. It remains to find a  point $y_n\in Y$ satisfying the conditions  (1a)--(1d).

\begin{lemma}\label{l2} For any nonempty set $J\subseteq (\w\times\w)\cap{\downarrow}\gamma$, integer number $l<\min\{\ell(i,j):(i,j)\in J\}$, and  nonempty open set $W\subseteq Y$ such that $W\cap \bigcap_{(i,j)\in J}Y_{\ell(i,j)-1}\ne\emptyset$, there exists a point $y\in W\cap Y_l\setminus Y_{l+1}$ such that $y\notin \bigcup_{(i,j)\in J}\overline{V}_{i,j}$.
\end{lemma}

\begin{proof}  Write the set $J$ as $\{(i_1,j_1),\dots,(i_m,j_m)\}$ for some pairs $(i_m,j_m)\prec\dots\prec(i_1,j_1)$. If $J$ is empty, then $m=0$. The inductive condition (2a) ensures that $\ell(i_k,j_k)+2\le \ell(i_{k-1},j_{k-1})$ for any $k\in\{2,\dots,m\}$. 

Let $v_1$ be any point in the set $W\cap Y_{\ell(i_1,j_1)-1}\setminus Y_{\ell(i_1,j_1)}$. Such point exists since the intersection $W\cap Y_{\ell(i_1,j_1)-1}$ is nonempty and $Y_{\ell(i_1,j_1)}$ is nowhere dense in $Y_{\ell(i_1,j_1)-1}$. 

By the inductive conditions (2c),(2f),(2a) we have $$\overline{V}_{\!i_1,j_1}=V_{i_1,j_1}\cup\partial V_{i_1,j_1}\subseteq (Y\setminus Y_{1+\ell(i_1)})\cup Y_{\ell(i_1,j_1)}$$ and $\ell(i_1)+2\le\ell(i_1,j_1)$. Hence, the set $Y_{\ell(i_1,j_1)-1}\setminus Y_{\ell(i_1,j_1)}$ is disjoint with $\overline{V}_{\!i_1,j_1}$. Then $v_1\notin  \overline{V}_{\!i_1,j_1}$ and we can choose an open neighborhood $W_1\subseteq W$ of $v_1$ such that $W_1\cap \overline{V}_{i_1,j_1}=\emptyset$.
\smallskip

Inductively we shall construct a sequence of points $v_2,\dots,v_m\in Y$ and a sequence of open sets $W_2,\dots,W_m$ in $Y$ such that for every $k\in\{2,\dots,m\}$ the following conditions are satisfied:
\begin{itemize}
\item[(i)] $v_k\in W_{k-1}\cap Y_{\ell(i_k,j_k)-1}\setminus Y_{\ell(i_k,j_k)}$;
\item[(ii)] $v_k\in W_k\subseteq W_{k-1}\setminus \overline{V}_{\!i_k,j_k}$.
\end{itemize}
Assume that for some $k\in\{2,\dots,m\}$ we have constructed  points 
$v_1,\dots,v_{k-1}$ and an open set $W_1,\dots,W_{k-1}$ satisfying the conditions (i), (ii). Since $v_{k-1}\in Y_{\ell(i_{k-1},j_{k-1})-1}\subseteq Y_{\ell(i_k,j_k)}$ and the set $Y_{\ell(i_k,j_k)}$ is nowhere dense in $Y_{\ell(i_k,j_k)-1}$, 
we can choose a point $v_k\in W_{k-1}\cap Y_{\ell(i_k,j_k)-1}\setminus Y_{\ell(i_k,j_k)}$.
By the inductive conditions (2c) and (2f), $\overline{V}_{\!i_k,j_k}=V_{i_k,j_k}\cup\partial V_{i_k,j_k}\subseteq (Y\setminus Y_{1+\ell(i_k)})\cup Y_{\ell(i_k,j_k)}$ and $\ell(i_k)+1\le \ell(i_k,j_k)-1$. Consequently, the set $Y_{\ell(i_k,j_k)-1}\setminus Y_{\ell(i_k,j_k)}$ is disjoint with $\overline{V}_{\!i_k,j_k}$. So, $v_k\notin  \overline{V}_{\!i_k,j_k}$ and we can choose an open neighborhood $W_k\subseteq W_{k-1}$ of $v_k$ such that $W_k\cap \overline{V}_{i_k,j_k}=\emptyset$.
This completes the inductive step.

After completing the inductive construction, consider the point $v_m\in Y_{\ell(i_m,j_m)-1}$ and its neighborhood $W_m\subseteq W$. The inductive condition (ii) guarantees that $$W_m\cap\bigcup_{k=1}^m\overline{V}_{\!i_k,j_k}=\bigcup_{k=1}^m(W_m\cap \overline{V}_{\!i_k,j_k})\subseteq \bigcup_{k=1}^m(W_k\cap \overline{V}_{\!i_k,j_k})=\emptyset.$$

Taking into account that $l<\min\{\ell(i,j):(i,j)\in J\}\le\ell(i_m,j_m)$, we conclude that $v_m\in Y_{\ell(i_m,j_m)-1}\subseteq Y_l$. Since the set $Y_{l+1}$ is nowhere dense in $Y_l$, there exists a point $y\in W_m\cap Y_l\setminus Y_{l+1}$. Since $W_m$ is disjoint with $\bigcup_{(i,j)\in J}\overline{V}_{\!i,j}$, the point $y_n$ does not belong to   
$\bigcup_{(i,j)\in J}\overline{V}_{i,j}$.
\end{proof}

Now we are able to find a $y_n$ satisfying the conditions (1a)--(1d).

Consider the sets $I(x_n)=\{(i,j)\in{\downarrow}\gamma:x_n\in U_{i,j}\}$ and $J(x_n)=\{(i,j)\in{\downarrow}\gamma:x_n\notin\overline{U}_{i,j}\}$. 

\begin{claim}\label{cl3}
\begin{enumerate}
\item For any $(i,j)\in I(x_n)\cup J(x_n)$ we have $\ell_X(x_n)<\ell(i,j)$.
\item For any $(i,j)\in I(x_n)$ we have $\ell_X(x_n)\le\ell(i)$.
\end{enumerate} 
\end{claim}

\begin{proof} 1.  If $(i,j)\in I(x_n)\cup J(x_n)$, then $x_n\notin \partial U_{i,j}=X_{\ell(i,j)}$ and hence $\ell_X(x_n)<\ell(i,j)$.

2. Now assume that $(i,j)\in I(x_n)$. Then $x_n\in U_{i,j}\subseteq X\setminus X_{1+\ell_X(x_i)}$ and hence $\ell_X(x_n)<1+\ell_X(x_i)=1+\ell(i)$ according to (1a).
\end{proof}

Choose a minimal subset $I\subseteq I(x_n)$ such that for every $(i,j)\in I(x_n)$ there exists $(p,q)\in I$ such that $U_{p,q}\subseteq U_{i,j}$. It is clear that $\bigcap_{(i,j)\in I(x_n)}U_{i,j}=\bigcap_{(i,j)\in I}U_{i,j}$.

\begin{claim}\label{cl3a} $\bigcap_{(i,j)\in I(x_n)}V_{i,j}=\bigcap_{(i,j)\in I}V_{i,j}$.
\end{claim}

\begin{proof} It suffices to show that for any $(i,j)\in I(x_n)$ there exists $(p,q)\in I$ such that $V_{p,q}\subseteq V_{i,j}$.
Given any pair $(i,j)\in I(x_n)$, find a  $(p,q)\in I$ such that $x_p\in U_{p,q}\subseteq U_{i,j}$ (such a pair exists by the choice of the set $I$).
If $(i,j)\prec p$, then $y_p\in V_{i,j}$ by the condition  (1c) and then $V_{p,q}\subseteq V_{i,j}$ by the condition (2d).

If $p\prec (i,j)$, the the condition (2b) implies that $p=i$ and condition (2d) ensures that $V_{p,q}\subseteq V_{i,j}$.
\end{proof}

\begin{claim}\label{cl5} For any pairs $(i,j)\prec(p,q)$ in $I$ we have  $\ell(p)\ge \ell(i,j)>\ell(i)$.
\end{claim}

\begin{proof} First we show that $x_p\in\partial U_{i,j}$. Assuming that $x_p\notin \partial U_{i,j}$, we conclude that $x_p\in U_{i,j}$ or $x_p\notin\overline{U}_{\!i,j}$. If $x_p\in U_{i,j}$, then the inductive condition (2d) guarantees that  $x_p\in U_{p,q}\subseteq U_{i,j}$ and the minimality of $I$ ensures that $(i,j)\notin  I$, which contradicts our assumption. So, $x_p\notin \overline{U}_{\!i,j}$. In this case, $U_{p,q}\cap \overline{U}_{\!i,j}=\emptyset$ by the condition (2e), but this contradicts $x_n\in U_{i,j}\cap U_{p,q}$.  Therefore, $x_p\in\partial U_{i,j}=X_{\ell(i,j)}$ and $\ell(p)=\ell_X(x_p)\ge \ell(i,j)>\ell_X(x_i)=\ell(i)$.
\end{proof}

Write the set $I$ as $\{(i_1,j_1),\dots,(i_m,j_m)\}$ for some pairs $(i_m,j_m)\prec\dots\prec(i_1,j_1)$. If $I$ is empty, then $m=0$. Claims~\ref{cl5} and \ref{cl3} imply  that $$\ell(i_1,j_1)>\ell(i_1)\ge \ell(i_2,j_2)>\ell(i_2)\ge\dots \ge\ell(i_m,j_m)>\ell(i_m)\ge \ell_X(x_n).$$
This chain of inequalities allows us to write the set $J(x_n)$ as the union $$J(x_n)=\Big(\bigcup_{k=1}^{m}J_k\Big)\cup\Big(\bigcup_{k=0}^mJ_k'\Big)$$ of the sets
$$
\begin{aligned}
&J_k=\{(i,j)\in J(x_n):\ell(i_k,j_k)>\ell(i,j)>\ell(i_k)\}\mbox{ \ for  $k\in\{1,\dots,m\}$,}\\
&J_0'=\{(i,j)\in J(x_n):\ell(i,j)>\ell(i_1,j_1)\},\\
&J_k'=\{(i,j)\in J(x_n):\ell(i_k)>\ell(i,j)>\ell(i_{k+1},j_{k+1})\}\mbox{ \ for  $k\in\{1,\dots,m-1\}$,}\\
&J'_m=\{(i,j)\in J(x_n):\ell(i_m)>\ell(i,j)>\ell_X(x_n)\}.
\end{aligned}
$$ 
Since the sets $\{\ell(i,j):(i,j)\in J(x_n)\}$ and $\{\ell(i_k),\ell(i_k,j_k)\}_{k=1}^m$ are disjoint, the union\break $\bigcup_{k=1}^{m}J_k\cup\bigcup_{k=0}^mJ_k'$ is indeed equal to $J(x_n)$.

By Lemma~\ref{l2}, there exists a point $v_0'\in Y_{\ell(i_1,j_1)}$ such that $v'_0\notin\bigcup_{(i,j)\in J'_0}\overline{V}_{i,j}$. Then $W'_0:=Y\setminus \bigcup_{(i,j)\in J'_0}\overline{V}_{i,j}$ is an open neighborhood of $v'_0$.

Inductively we shall construct a sequence of points $v_1,v'_1,,\dots,v_m,v'_m$ and a sequence of open sets $W_1\supseteq W_1'\supseteq\dots\supseteq W_m\supseteq W_m'$ in $Y$ such that for every $k\in\{1,\dots,m\}$ the following conditions are satisfied:
 \begin{enumerate}
 \item[(a)] $v_{k}\in W'_{k-1}\cap V_{i_k,j_k}\cap Y_{\ell(i_k)}$;
 \item[(b)] $v_k\in W_k=W'_{k-1}\cap V_{i_k,j_k}$;
 \item[(c)] $W_k\cap\bigcup_{(i,j)\in J_k}\overline{V}_{\!i,j}=\emptyset$;
 \item[(d)] $v'_k\in W_k'\subseteq W_k\setminus\bigcup_{(i,j)\in J'_k}\overline{V}_{\!i,j}$; 
 \item[(e)] if $k<m$, then $v'_k\in Y_{\ell(i_{k+1},j_{k+1})}$;
 \item[(f)] $v'_m\in Y_{\ell_X(x_n)}\setminus Y_{1+\ell_X(x_n)}$.
\end{enumerate}
 
To make an inductive step, assume that for some $k\in\{1,\dots,m\}$ a point $v'_{k-1}$ and an open set $W'_{k-1}$ with $v_{k-1}'\in W_{k-1}\cap Y_{\ell(i_k,j_k)}$ have been constructed. By the inductive condition (2f), $Y_{\ell(i_k,j_k)}=\partial V_{i_k,j_k}\subseteq \overline{ V_{i_k,j_k}\cap Y_{\ell(i_k)}}$. Consequently, there exists a point $v_k\in W'_{k-1}\cap  V_{i_k,j_k}\cap Y_{\ell(i_k)}$. 
Put $W_k:=W_{k-1}\cap V_{i_k,j_k}$. It is clear that the inductive conditions (a), (b) are satisfied.

\begin{claim}\label{cl5a} $V_{i_k,j_k}\cap\bigcup_{(i,j)\in J_k}\overline{V}_{\!i,j}=\emptyset$.
\end{claim}

\begin{proof} To derive a contradiction, assume that $V_{i_k,j_k}\cap\overline{V}_{i,j}\ne\emptyset$ for some $(i,j)\in J_k$. The definition of the set $J_k\ni (i,j)$ yields  $\ell(i_k,j_k)>\ell(i,j)>\ell(i_k)$ and hence $(i,j)\prec(i_k,j_k)$ by the condition (2a). It follows from (2e) and $V_{i_k,j_k}\cap \overline V_{\!i,j}\ne\emptyset$ that $y_{i_k}\in \overline V_{\!i,j}$ and hence $x_{i_k}\in\overline U_{\!i,j}$ according to the condition (1d).  Assuming that $x_{i_k}\in U_{i,j}$, we obtain $U_{i_k,j_k}\subseteq U_{i,j}$ by the inductive condition (2d). Then $x_n\in U_{i_k,j_k}\subseteq U_{i,j}$, which contradicts the inclusion $(i,j)\in J_k$. Therefore, $x_{i_k}\notin U_{i,j}$ and hence $y_{i_k}\notin V_{i,j}$ by condition (1c). Then $y_{i_k}\in\overline V_{i,j}\setminus V_{i,j}=\partial V_{i,j}$ and hence $\ell(i_k)=\ell_Y(y_{i_k})\ge\ell(i,j)$, which contradicts the inclusion $(i,j)\in J_k$.
\end{proof}

Claim~\ref{cl5a} and the condition (b) imply the condition (c).
\smallskip

Since $v_k\in W_k\cap Y_{\ell(i_k)}\subseteq W_k\cap\bigcap_{(i,j)\in J_k'}Y_{\ell(i,j)-1}$, we can apply Lemma~\ref{l2} and find a point $v_k'\in W_k$ and a neighborhood $W_k'$ of $v_k'$ satisfying the inductive conditions (d),(e),(f).

After completing the inductive construction, we conclude that the open subset $W'_m\cap Y_{\ell(n)}\setminus Y_{1+\ell(n)}$ of the crowded space $Y_{\ell(n)}$ contains the point $v_m'$ and hence is not empty. Since the space $Y_{\ell(n)}$ is crowded (see Lemma~\ref{l:supercrowded}) and the set $\bar B\cap Y_{\ell(n)}$ is nowhere dense in $Y_{\ell(n)}$, there exists a point $$y_n\in 
W_m'\cap (Y_{\ell(n)}\setminus Y_{1+\ell(n)})\setminus(\bar B\cup \{y_k\}_{k\in{\downarrow}\gamma}).$$ The inductive conditions (a),(c),(d) and Claim~\ref{cl3a} imply that $I(x_n)\subseteq I(y_n)$ and $J(x_n)\subseteq J(y_n)$, where
$$I(y_n)=\{(i,j)\in{\downarrow}\gamma:y_n\in V_{i,j}\}\mbox{ \ and \ }
 J(y_n)=\{(i,j)\in{\downarrow}\gamma:y_n\notin \overline{V}_{i,j}\}.$$
 The condition (1c) will follow as soon as we show that $I(x_n)=I(y_n)$. Assuming that $I(x_n)\ne I(y_n)$, we can find a pair $(i,j)\in I(y_n)\setminus I(x_n)$. The inclusion $J(x_n)\subseteq J(y_n)\subseteq {\downarrow}\gamma\setminus I(y_n)$ implies that $(i,j)\notin J(x_n)$ and hence $x_n\in \overline U_{\!i,j}\setminus U_{i,j}=X_{\ell(i,j)}$. Then $\ell_Y(y_n)=\ell_X(x_n)\ge \ell(i,j)$ and hence $y_n\in Y_{\ell(i,j)}=\partial V_{i,j}$ and hence $y_n\notin V_{i,j}$ and  $(i,j)\notin I(y_n)$, which contradicts the choice of $(i,j)$. This completes the proof of condition (1c).
 
To prove the condition (1d), assume that $J(x_n)\ne J(y_n)$ and find a pair $(i,j)\in J(y_n)\setminus J(x_n)$. Then $x_n\in \overline U_{\!i,j}$. Assuming that $x_n\in U_{i,j}$, we conclude that $(i,j)\in I(x_n)=I(y_n)$ and hence $y_n\in V_{i,j}\subseteq \overline V_{\!i,j}$, which contradicts $(I,j)\in J(y_n)$. Therefore, $x_n\in\overline U_{\!i,j}\setminus U_{i,j}=\partial U_{i,j}=X_{\ell(i,j)}$ and $\ell_Y(y_n)=\ell_X(x_n)\ge \ell(i,j)$ and then the condition (2f) ensures that $y_n\in Y_{\ell(i,j)}=\partial V_{i,j}$ and hence $(i,j)\notin J(y_n)$, which contradicts the choice of the pair $(i,j)$. This contradiction completes the proof of the condition (1d). It is clear that the conditions (1e)--(1g) holds. 
\smallskip

$1'''$. $n\in\overleftarrow\Omega$. In this case the set $\overleftarrow\Omega$ is not empty and hence $(X_n)_{n\in\w}$ is a canonical superskeleton for $X$, the set $B$ is closed in $Y$ and for every $n\in\w$ the set $A\cap X_n$ is nowhere dense in $X_n$. In this case we put $y_n=\min(Y'\setminus \{y_k\}_{k\in{\downarrow}\gamma})$ and repeating the argument from the case $1''$, can find a point $x_n\in X$ satisfying the conditions (1a)--(1g).
\smallskip

2. Now consider the second case: $\gamma=(n,k)$ for some $(n,k)\in\w\times\w$. Since $n\prec (n,k)$, the points $x_n,y_n$ have been already defined. So, we can choose open sets $U\subseteq X$ and $V\subseteq Y$ such that 
\begin{itemize}
\item $x_k\notin \overline U\ne X$ and $y_k\notin\overline V\ne Y$ for every $k\in{\downarrow}\gamma\setminus\{x_n\}$, 
\item $x_n\in U\subseteq O_k^X(x_n)\cap\bigcap_{(i,j)\in I(x_n)}U_{i,j}\setminus\bigcup_{(i,j)\in J(x_n)}\overline U_{\!i,j}$,\quad and 
\item $y_n\in V\subseteq O_k^Y(y_n)\cap\bigcap_{(i,j)\in I(y_n)}V_{i,j}\setminus\bigcup_{(i,j)\in J(y_n)}\overline V_{\!i,j}$,
\end{itemize}
where
$$
\begin{gathered} 
 I(x_n)=\{(i,j)\in{\downarrow}\gamma:x_n\in U_{i,j}\},\quad 
  I(y_n)=\{(i,j)\in{\downarrow}\gamma:y_n\in V_{i,j}\}\\
J(x_n)=\{(i,j)\in{\downarrow}\gamma:x_n\notin \overline U_{\!i,j}\},\quad J(y_n)=\{(i,j)\in{\downarrow}\gamma:y_n\notin \overline V_{\!i,j}\}.
\end{gathered}
$$
If $n\notin\Omega$, then $x_n\notin A$, $y_n\notin \overline B$ (by the inductive conditions (1f), (1g)) and we can (and will)  additionally assume that $$U\cap A=\emptyset=V\cap \bar B.$$

Since $(X_i)_{i\in\w}$ is a coregular skeleton for the space $X$ and $(Y_i)_{i\in\w}$ is a superskeleton for the space $Y$, there exists a number $l\ge 2+\max\{\ell(\alpha):\alpha\in{\downarrow}\gamma\}$ such that $$X_l\subseteq \overline U\cap\overline{X\setminus\overline U}\subseteq\partial U\mbox{ \ and  \ }Y_l\subseteq \overline {V\cap Y_{\ell(n)}}\cap\overline{Y\setminus\overline V}\subseteq\partial V.$$

Since the skeleton $(Y_i)_{i\in\w}$ is coregular, the complement $Y\setminus Y_l$ is a regular topological space. Being  second-countable, the regular space $Y\setminus Y_l$ is metrizable (by the Urysohn Metrization Theorem \cite[4.2.9]{Eng}). Being countable, the metrizable space $Y\setminus Y_l$ is zero-dimensional. Then we can find a closed-and-open neighborhood $V'\subseteq Y\setminus Y_l$ of the point $y_n$ such that $V'\subseteq V$. Then $\partial V'\subseteq Y_l$. 

By analogy we prove that the space $X\setminus X_l$ is metrizable and zero-dimensional.
By Theorem~\cite[7.1.11]{Eng}, the countable zero-dimensional space $X\setminus X_l$ is strongly zero-dimensional (which means that any disjoint closed sets in $X\setminus X_l$ can be separated by closed-and-open neighborhoods). Observe that the sets $f^{-1}(V')$ and $f^{-1}((Y\setminus Y_l)\setminus V')$ are two closed disjoint sets in $A\setminus X_l$ and $X\setminus X_l$. By the strong zero-dimensionality of $X\setminus X_l$, there exists a closed-and open set $U'$ in $X\setminus X_l$ such that $$\{x_n\}\cup f^{-1}(V')\subseteq U'\subseteq  U\setminus  f^{-1}((Y\setminus Y_l)\setminus V').$$
Then $\partial U'\subseteq X_l$ and $f(A\cap U')=B\cap V'$.

Since $(X_i)_{i\in\w}$ is a coregular skeleton for $X$ and $(Y_i)_{i\in\w}$ is a superconnecting skeleton for $Y$, there exists a number $p>l$ such that $X_p\subseteq \partial U'$ and $Y_p\subseteq \partial V'$.

\begin{lemma}\label{l:vnk} There exists closed-and open subset $V_{n,k}\subseteq Y\setminus Y_p$ such that $V'\subseteq V_{n,k}\subseteq V$, $(V_{n,k}\setminus V')\cap \bar B=\emptyset$ and $\partial V_{n,k}=Y_l\subseteq\overline{V_{n,k}\cap Y_{\ell(n)}}$.
\end{lemma}

\begin{proof} By the Urysohn Metrization Theorem \cite[4.2.9]{Eng}, the second-countable regular space $Y\setminus Y_p$ is metrizable. So, we can find a  metric $d$ generating the topology of  $Y$.
Since the set $Y_l\setminus Y_{p}\subseteq Y$ is countable and nonempty, there exists a function $\hbar:\w\to Y_l\setminus Y_{p}$ such that for every $y\in Y_l\setminus Y_{p}$ the preimage $\hbar^{-1}(y)$ is infinite. Since $Y_l\setminus Y_p\subseteq \overline{V\cap Y_{\ell(n)}}$ and the set $\bar B\cap Y_{\ell(n)}$ is nowhere dense in $Y_{\ell(n)}$, for every $m\in\w$ we can find a point $v_m\in V\cap Y_{\ell(n)}\setminus \bar B$ such that $d(v_m,\hbar(m))<2^{-m}$.
Since the space $Y\setminus Y_p$ is zero-dimensional, the point $v_m$ has a closed-and-open neighborhood $W_m$ in $Y\setminus (Y_p\cup\bar B)$  such that $W_k\subseteq V\cap\{y\in Y\setminus Y_p:d(y,v_m)<2^{-m}\}$. Then the boundary $\partial W_m$ of $W_m$ in $Y$ is contained in $Y_p$. We claim that the open neighborhood $$V_{n,k}=V'\cup\bigcup_{m\in\w}W_m\subseteq V$$ of $y_n$ has the required property: $X_l=\partial V_{n,k}\subseteq\overline{V_{n,k}\cap Y_{\ell(n)}}$.

First we show that $Y_l\subseteq \overline{V_{n,k}\cap Y_{\ell(n)}}$. Given any point $a\in Y_l$ and open neighborhood $O_a\subseteq Y$ of $a$,  use the nowhere density of $Y_p$ in $Y_l$ and find a point $b\in O_a\cap (Y_l\setminus Y_{p})$. Since the metric $d$ generates the topology of the space $Y\setminus Y_p$, there exists a number $q\in\w$ such that the ball $B(b;2^{-q})=\{y\in Y\setminus Y_l:d(y,b)<2^{-q}\}$ is contained in $O_a$. Since the set $\hbar^{-1}(b)$ is infinite, there exists $m>q$ such that $\hbar(m)=b$. Then $d(b,v_m)=d(\hbar(m),v_m)<2^{-m}<2^{-q}$ and hence $v_m\in O_a$. On the other hand, $v_m\in V_{n,k}\cap Y_{\ell(n)}$, witnessing that $O_a\cap (V\cap Y_{\ell(n)})\ne\emptyset$ and hence $a\in \overline{V_{n,k}\cap Y_{\ell(n)}}\subseteq \overline V_{\!n,k}$. 

On the other hand, $a\in Y_l\subseteq Y_p\subseteq \overline{X\setminus \overline V}\subset \overline{Y\setminus V_{n,k}}$ and hence $a\in \overline{V_{n,k}}\cap\overline{Y\setminus V_{n,k}}=\partial V_{n,k}$. Therefore, $Y_l\subseteq \partial V_{n,k}$. Assuming that $Y_l\ne \partial V_{n,k}$, we can find a point $z\in \partial V_{n,k}\setminus Y_l$. Choose $s\in\w$ such that the ball $B(z;2^{-s})=\{y\in Y:d(z,y)<2^{-s}\}$ does not intersect the closed subset $Y_l\setminus Y_p$ of $Y\setminus Y_p$. We claim that for every $m\ge s+2$, the ball $B(z;2^{-s-1})$ does not intersect the set $W_m$. Assuming that $B(z;2^{-s-1})\cap W_m$ contains some point $w$, we conclude that 
\begin{multline*}
d(z,\hbar(m))\le d(z,w)+d(w,v_m)+d(v_m,\hbar(m))<\\
2^{-s-1}+2^{-m}+2^{-m}=2^{-s-1}+2^{-m+1}\le 2^{-s-1}+2^{-s-1}=2^{-s},
\end{multline*}
which contradicts the choice of $s$.
Since $z\in \partial V_{n,k}$, $z\notin V_{n,k}$ and hence $z\notin V'\cup \bigcup_{m<2+s}W_m$. It follows from $\partial V'\cup\bigcup_{m<2+s}\partial W_m\subseteq Y_{l}\not\ni z$ that $z\notin \overline {V'}\cup\bigcup_{m<2+s}\overline W_m$. Then we can find a neighborhood $O_z\subseteq B(z;2^{-s-1})$ such that $O_z\cap \overline{V'}\cup\bigcup_{m<2+s}\overline W_{\!m}=\emptyset$ and hence
$$O_z\cap V_{n,k}\subseteq \big(O_z\cap V'\big)\cup\Big(\bigcup_{m<2+s}O_z\cap W_m\Big)\cup\Big(\bigcup_{m\ge 2+s}B(s;2^{-s-1})\cap W_m\Big)=\emptyset,$$
which contradicts $z\in\partial V_{n,k}$. This contradiction shows that $\partial V_{n,k}=Y_l\subseteq \overline V_{\!n,k}\cap Y_{\ell(n)}$. The choice of the sets $W_m\subseteq Y\setminus \bar B$ ensures that 
$$(V_{n,k}\setminus V)\cap \bar B\subseteq\bigcup_{m\in\w}(W_m\cap \bar B)=\emptyset.$$  
\end{proof}

By analogy we can prove the following lemma.

\begin{lemma}\label{l:unk} There exists closed-and-open subset $U_{n,k}\subseteq X\setminus X_p$ such that $U'\subseteq U_{n,k}\subseteq U$, $(U_{n,k}\setminus U')\cap A=\emptyset$ and $\partial U_{n,k}=X_l$. Moreover, if $\overleftarrow\Omega\ne\emptyset$, then $X_l\subseteq\overline{U_{n,k}\cap X_{\ell(n)}}$.
\end{lemma}

Finally, observe that the number $\ell(n,k):=l$ and the sets $U_{n,k}$ and $V_{n,k}$ constructed in Lemmas~\ref{l:vnk} and \ref{l:unk} satisfy the conditions (2a)--(2i).
 This completes the inductive step.
 \smallskip
 
 After completing the inductive construction, observe that the inductive condition (1f)  implies that $X=\{x_n\}_{n\in\w}$. So, we can consider the map $\bar f:X\to Y$ such that $\bar f(x_n)=y_n$ for every $n\in\w$. The inductive condition (1e) ensures that $\bar f{\restriction}A=f$.  
 We claim that the map $\bar f$ is a topological embedding.
 
 To see that $f$ is continuous, take any $n\in\w$ and any neighborhood $O(y_n)$ of the point $y_n=f(x_n)$. Find $k\in\w$ such that $O^Y_k(y_n)\subseteq O(y_n)$. The inductive condition (2b) guarantees that $V_{n,k}\subseteq O^Y_k(y_n)$. We claim that $\bar f(U_{n,k})\subseteq V_{n,k}$. Indeed, for any $x_m\in U_{n,k}\setminus\{x_n\}$, the inductive condition (2b) ensures that $(n,k)\prec m$. Then $y_m\in V_{n,k}$ by the condition (1c). Therefore, $f(U_{n,k})\subseteq V_{n,k}\subseteq O^Y_k(y_n)\subseteq O(y_n)$, witnessing that the map $f$ is continuous. By analogy we can prove the continuity of the map $f^{-1}:f(X)\to X$.
 
If $(X_n)_{n\in\w}$ is a canonical supserskeleton for $X_n$, the set $B$ is closed in $Y$, and for every $n\in\w$ the set $A\cap X_n$ is nowhere dense in $X_n$, then the set $\overleftarrow\Omega$ is infinite and the inductive condition (1g) implies that $\bar f(X)=\{y_n\}_{n\in\w}=Y$ and hence the topological embedding $\bar f$ is a homeomorphism.
 \end{proof}
 
 Now we deduce some corollaries of Theorem~\ref{t:main}.
 
\begin{corollary}\label{c1} Let $X,Y$ be two countable second-countable topological spaces, $(X_n)_{n\in\w}$ and $(Y_n)_{n\in\w}$ be canonical superskeleta  in the spaces $X,Y$, and $A,B$ be closed nowhere dense sets in the spaces $X,Y$, respectively. 
Let $h:A\to B$ be a homeomorphism such that $h(A\cap X_n)=(B\cap Y_n)$ for every $n\in\w$. Then there exists a homeomorphism $\bar h:X\to Y$ such that $\bar h{\restriction}A=h$ and $h(X_n)=Y_n$ for all $n\in\w$.
\end{corollary}

\begin{corollary}\label{c2}  Let $X,Y$ be two countable second-countable topological spaces and $(X_n)_{n\in\w}$ and $(Y_n)_{n\in\w}$ be canonical superskeleta  in the spaces $X,Y$, respectively. Then there exists a homeomorphism $h:X\to Y$ such that  $h(X_n)=Y_n$ for all $n\in\w$.
\end{corollary}

Lemma~\ref{l:nodense} and Corollary~\ref{c2} imply another corollary.

\begin{corollary}\label{c3} Let $X,Y$ be two countable second-countable topological spaces and $(X_n)_{n\in\w}$ and $(Y_n)_{n\in\w}$ be superskeleta  in the spaces $X,Y$, respectively. Then there exists an increasing number sequence $(n_k)_{k\in\w}$ and a homeomorphism $h:X\to Y$ such that  $h(X_{n_k})=Y_{n_k}$ for all $k\in\w$.
\end{corollary}

Now we are able to prove Theorem~\ref{t:main} reformulating it as follows.

\begin{theorem}[Characterization of $\IQ\mathsf P^\infty$]\label{t:chara} A topological space $X$ is homeomorphic to the space $\IQ\mathsf P^\infty$ if and only if $X$ is countable second-countable and admits a vanishing sequence $(X_n)_{n\in\w}$ of nonempty closed sets that has two properties:
\begin{enumerate}
\item for every $n\in\w$ and a nonempty open set $U\subseteq X_n$ the closure $\overline{U}$ contains some set $X_m$;
\item for every $n\in\w$ the complement $X\setminus X_n$ is a regular topological space.
\end{enumerate}
\end{theorem}

\begin{proof} The ``only if'' part follows from Theorem~\ref{p3}. To prove the ``if'' part, assume that the space $X$ is countable, second-countable and $X$ has a vanishing sequence of nonempty closed sets $(X_n)_{n\in\w}$ satisfying the conditions (1),(2). By Definition~\ref{d:skeleton}, $(X_n)_{n\in\w}$ is a superskeleton for $X$. By Theorem~\ref{p3}, the space $\IQ\mathsf P^\infty$ also is countable, second-countable and has a superskeleton. By Corollary~\ref{c3}, the spaces $X$ and $\IQ\mathsf P^\infty$ are homeomorphic.
\end{proof}

Now we prove a universality property of the space $\IQ\mathsf P^\infty$.

\begin{theorem}[Universality of $\IQ\mathsf P^\infty$]\label{t:univ} Each countable second-countable coregular space $X$ is homeomorphic to a subspace of $\IQ\mathsf P^\infty$.
\end{theorem}

\begin{proof} By Lemma~\ref{l:cr}, the space $X$ admits a coregular skeleton $(X_n)_{n\in\w}$. By Theorem~\ref{p3}, the space $\IQ\mathsf P^\infty$ has a canonical superskeleton $(Y_n)_{n\in\w}$. Applying Theorem~\ref{t:main} with $A=B=\emptyset$, we obtain a topological embedding $f:X\to Y$ such that $f^{-1}(Y_n)=X_n$ for all $n\in\w$.
\end{proof}

Now we prove a (rather strong) homogeneity property of the space $\IQ\mathsf P^\infty$.

A subset $A$ of a  topological space $X$ is called
\begin{itemize}
\item {\em deep} if for any  non-empty open sets $U_1,\dots,U_n\subseteq X$ the set $A\setminus (\overline U_{\!1}\cap\dots\cap\overline U_{\!n})$ is finite.
\item {\em shallow} if there exist  non-empty open sets $U_1,\dots,U_n\subseteq X$ such that $A\cap  (\overline U_{\!1}\cap\dots\cap\overline U_{\!n})=\emptyset$.
\end{itemize}
This definition implies that  for any deep (resp. shallow) set $A$ in a topological space $X$ and any homeomorphism $h:X\to X$ the set $h(A)$ is deep (resp. shallow). Observe also that any infinite set in a second-countable space contains an infinite subset which is either deep or shallow. The definition implies that any finite set in a Hausdorff space is shallow.

\begin{theorem}[Dychotomic Homogeneity of $\IQ\mathsf P^\infty$]\label{t:hom}
Let $A,B$ be two closed discrete subsets of $\IQ\mathsf P^\infty$. If the sets $A,B$ are either both deep or both shallow, then any bijection $f:A\to B$ extends to  a homeomorphism $h$ of $\IQ\mathsf P^\infty$ such that $h(A)=B$.
\end{theorem}

\begin{proof}  By Theorem~\ref{p3}, the space $\IQ\mathsf P^\infty$ has a canonical superskeleton $(X_n)_{n\in\w}$. If both sets $A,B$ are shallow, then we can find a number $m\in\w$ such that $X_m$ is disjoint with the set $A\cup B$. Let $Y_0=X_0$ and $Y_n=X_{m+n}$ for $n\in\IN$. Observe that $(Y_n)_{n\in\w}$ is a canonical superskeleton for the space $\IQ\mathsf P^\infty$ such that $A\cup B\subseteq Y_0\setminus Y_1$. Since the space $\IQ\mathsf P^\infty$ is crowded, the sets $A,B$ are nowhere dense in $X$. Applying Corollary~\ref{c1}, we can find a homeomorphism $h:X\to X$ such that $h{\restriction}A=f$ and $h(Y_n)=Y_n$ for all $n\in\w$.

The case of deep sets $A,B$ is more tricky. Let $\ell:X\to\w$ be the function assigning to each point $x\in X$ the unique number $n\in\w$ such that $x\in X_n\setminus X_{n+1}$. The deepness of $A,B$ implies that for every $n\in\w$ the set $(A\cup B)\setminus X_n$ is finite. Choose an increasing sequence $(n_k)_{k\in\w}$ such that $n_0=0$ and for every $k\in\w$ the following conditions hold: 
\begin{itemize}
\item $n_{k+1}>\ell(f(a))$ for any $a\in A\setminus X_{n_k+1}$;
\item $n_{k+1}>\ell(f^{-1}(b))$ for any $b\in B\setminus X_{n_k+1}$.
\end{itemize}
For every $k\in \w$ consider the finite sets 
$$
\begin{aligned}
&A_k=A\cap X_{n_k}\setminus X_{n_{k+1}},\quad &&B_k=B\cap X_{n_k}\setminus X_{n_{k+1}},\\
&A_k^{=}=\{a\in A_k:f(a)\in X_{n_k}\setminus X_{n_{k+1}}\},\quad &&B_k^{=}=\{b\in B_k:f^{-1}(b)\in X_{n_k}\setminus X_{n_{k+1}}\},\\
&A_k^{+}=\{a\in A_k:f(a)\in X_{n_{k+1}}\},\quad &&B_k^{+}=\{b\in B_k:f^{-1}(b)\in X_{n_{k+1}}\},\\
&A_k^{-}=\{a\in A_k:f(a)\notin X_{n_k}\},\quad &&B_k^{-}=\{b\in B_k:f^{-1}(b)\notin 
X_{n_k}\}.
\end{aligned}
$$
\begin{claim} For any $k\in\w$ we have $$f(A_k^=)=B_k^=,\;\;f(A_k^+)=B^-_{k+1},\;\;f^{-1}(B_k^+)=A_{k+1}^-, \mbox{ \  and \ }f(A^-_{k+1})=B^+_k.$$
\end{claim}

\begin{proof} The equality $f(A_k^=)=B_k^=$ follows from the definition of the sets $A_k^=$ and $B_k^=$. To show that $f(A_k^+)=B^-_{k+1}$, take any $a\in A_k^+$. Then $f(a)\in B\cap X_{n_{k+1}}$ by the definition of $A_k^+$. The definition of the number $n_{k+2}$ guarantees that $\ell(f(a))<n_{k+2}$ and hence 
$$f(a)\in B\cap X_{n_{k+1}}\setminus X_{n_{k+2}}=B_{k+1}.$$ Since $a=f^{-1}(f(a))\notin X_{n_{k+1}}$, the point $f(a)$ belongs to $B_{k+1}^-$. Therefore, $f(A_k^+)\subseteq B^-_{k+1}$. Now take any point $b\in B_{k+1}^-$ and observe that the point $a=f^{-1}(b)$ dos not belong to $X_{n_{k+1}}$ by the definition of the set $B_{k+1}^-$. Assuming that $a\notin X_{n_k}$, we conclude that $\ell(b)=\ell(f(a))<n_{k+1}$ and hence $b\notin X_{n_{k+1}}$, which contradicts the choice of $b$. This contradiction shows that $a\in A\cap X_{n_k}\setminus X_{n_{k+1}}$. Since $b=f(a)\in X_{n_{k+1}}$, the point $a$ belongs to the set $A^+_{k}$ and hence $b\in f(A_k^+)$. Therefore, $f(A_k^+)=B_{k+1}^-$. By analogy we can prove that $f^{-1}(B^+_k)=A^-_{k+1}$ and hence $f(A^-_{k+1})=B^+_k$.
\end{proof}

\begin{claim} For every $k\in\w$ we have $A_k^+\cup B_k^+\subseteq X_{n_k+1}\setminus X_{n_{k+1}}$.
\end{claim}

\begin{proof} Assuming that $A_k^+\not\subseteq X_{n_{k+1}}$, we can find a point $a\in A_k^+\setminus X_{n_k+1}$. The choice of $n_{k+1}$ ensures that $n_{k+1}>\ell(f(a))$ and hence $f(a)\notin X_{n_{k+1}}$ , which contradicts the inclusion $a\in A_k^+$. This contradiction shows that  $A_k^+\subseteq X_{n_k+1}\setminus X_{n_{k+1}}$. By analogy we can prove that $B_k^+\subseteq X_{n_k+1}\setminus X_{n_{k+1}}$.
\end{proof} 

The coregularity of the skeleton $(X_n)_{n\in\w}$ guarantees that for every $k\in\w$ the countable space $X_{n_k+1}\setminus X_{n_{k+1}}$ is regular and hence metrizable and zero-dimensional. Then we can find a closed-and-open sets $U_k$ and $V_k$ in $X_{n_k+1}\setminus X_{n_{k+1}}$ such that $A_k\cap U_k=A_k^+$ and $B_k\cap V_k=B_k^+$. 

Let $Y_0=Z_0=X_0$ and for every $k\in\IN$ let $$Y_k:=U_k\cup X_{n_{k+1}}\quad\mbox{and}\quad Z_k:=V_k\cup X_{n_{k+1}}$$and observe that 
$Y_k=U_k\cup X_{n_{k+1}}\subseteq X_{n_k+1}\subset X_{n_k}\subseteq Y_{k-1}$. The nowhere density of the set $X_{n_k+1}$ in $X_{n_k}$ implies the nowhere density of $Y_k$ in $X_{n_k}$ and also in $Y_{k-1}$. By analogy we can show that $Z_k$ is nowhere dense in $Z_{k-1}$. It is easy to check that $(Y_k)_{k\in\w}$ and $(Z_k)_{k\in\w}$ are canonical superskeleta for $X$ such that $$A\cap Y_k\setminus Y_{k+1}=A_{k+1}^-\cup A_{k+1}^=\cup A_k^+\quad\mbox{ and }\quad B\cap Z_k\setminus Z_{k+1}=B_{k}^+\cup B_{k+1}^=\cup B_{k+1}^-$$for every $k\in\w$. This implies that $f(A\cap Y_k)=B\cap Z_k$ for every $k\in\w$. Since the spaces $A,B$ are discrete, the intersections $A\cap Y_k$ and $B\cap Z_k$ are nowhere dense in the crowded spaces $Y_k,Z_k$, respectively. Applying Theorem~\ref{t:main}, we can find a homeomorphism $h:X\to Y$ such that $h{\restriction}A=f$ and $h(Y_k)=Z_k$ for all $k\in\w$.
\end{proof}

 Since any finite set in a Hausdorff topological space is shallow, Theorem~\ref{t:hom} implies that the following finite homogeneity property of the space $\IQ\mathsf P^\infty$.

\begin{corollary}[Finite homogeneity of $\IQ\mathsf P^\infty$]
Any bijective function $f:A\to B$ between two finite subsets $A,B$ of the space $\IQ\mathsf P^\infty$ can be extended to a homeomorphism of $\IQ\mathsf P^\infty$.
\end{corollary}

\begin{remark} By Lemma~\ref{l:product}, the product $\IQ\mathsf P^\infty\times\IQ\mathsf P^\infty$ is not coregular and hence cannot be homeomorphic to $\IQ\mathsf P^\infty$. Nonetheless, $\IQ\mathsf P^\infty\times\IQ\mathsf P^\infty$ contains a dense subspace homeomorphic to $\IQ\mathsf P^\infty$, see \cite{Stel}.
\end{remark}


\section{Acknowledgement}

The first author would like to thank MathOverflow user Fedor Petrov\footnote{\tt mathoverflow.net/q/286366} who turned his attention to the infinite projective space $\IQ\mathsf P^\infty$ and informed the authors about the paper \cite{GF} of Gelfand and Fuks.


\end{document}